\documentclass[11pt,a4paper,oneside]{article}
\usepackage[english]{babel}
\usepackage[T1]{fontenc}
\usepackage{amsmath}
\usepackage{amsthm}
\usepackage{amsfonts}
\usepackage{amssymb}
\usepackage{lmodern} 
\usepackage{indentfirst}		
\usepackage{microtype} 			
\usepackage[numbers]{natbib}
\usepackage{hyphenat}
\usepackage{pgfplots}  	
\usepackage[latin1]{inputenc}
\usepackage{indentfirst}

\definecolor{blackgreen}{RGB}{0,80,0}

\usepackage[plainpages=false,pdfpagelabels,pdfusetitle,pdfdisplaydoctitle, pdfstartpage=1, pdfstartview={{XYZ null null 1.00}}]{hyperref}
\hypersetup{
	pdftitle={GWP and critical norm concentration for IBNLS},    
	pdfauthor={Mykael Cardoso,  Carlos Guzm\'an and Ademir Pastor},    
	pdfnewwindow=true,      
	colorlinks=true,      
	linkcolor=blue,         
	citecolor=red}
\usepackage[top=2cm, bottom=2cm, left=2cm, right=2cm]{geometry}
\usepackage{mathtools}
\usepackage{mathabx}
\mathtoolsset{showonlyrefs}

\newcommand{\Real}{\mathbb R}
\newtheorem{thm}{Theorem}[section]
\newtheorem{prop}[thm]{Proposition}
\newtheorem{lemma}[thm]{Lemma}
\newtheorem{rem}[thm]{Remark}

\newtheorem{coro}[thm]{Corollary}
\numberwithin{equation}{section}

\pgfplotsset{compat=1.16}

\title{Global well-posedness and critical norm concentration for inhomogeneous biharmonic NLS}
\author{Mykael Cardoso, Carlos M. Guzm\'an and Ademir Pastor} 
\date{} 

\begin{document}

\maketitle

\begin{abstract}\noindent
We consider the inhomogeneous biharmonic nonlinear Schr\"odinger (IBNLS) equation in $\mathbb{R}^N$,
$$i \partial_t u +\Delta^2 u -|x|^{-b} |u|^{2\sigma}u = 0,$$
where $\sigma>0$ and $b>0$. We first study the local well-posedness in $\dot H^{s_c}\cap \dot H^2 $, for $N\geq 5$ and $0<s_c<2$, where $s_c=\frac{N}{2}-\frac{4-b}{2\sigma}$. Next, we established a Gagliardo-Nirenberg type inequality in order to obtain sufficient conditions for global existence of solutions in $\dot H^{s_c}\cap \dot H^2$ with $0\leq s_c<2$. Finally, we study the phenomenon of $L^{\sigma_c}$-norm concentration for finite time blow up solutions with bounded $\dot H^{s_c}$-norm, where $\sigma_c=\frac{2N\sigma}{4-b}$. Our main tool is the compact embedding of $\dot L^p\cap \dot H^2$ into a weighted $L^{2\sigma+2}$ space, which may be seen of independent interest.
\end{abstract}

{\small

}

\section{Introduction}
In this paper, we study the initial value problem (IVP) for the focusing inhomogeneous biharmonic nonlinear Schr\"odinger (IBNLS) equation
\begin{equation}
\begin{cases}\label{PVI}
i \partial_t u + \Delta^2 u - |x|^{-b} |u|^{2 \sigma}u = 0, \,\,\, x \in \mathbb{R}^N, \,t>0,\\
u(\cdot,0) = u_0 ,
\end{cases}
\end{equation}
where $\sigma$, $b>0$ and $u = u(t, x)$ is a complex-valued function in space-time $\Real^N \times \Real$. Here, $\Delta^2$ stands for the biharmonic operator, that is, $\Delta^2u=\Delta(\Delta u)$. Equation in \eqref{PVI} may be seen as an inhomogeneous version of the fourth order NLS equation,
\begin{equation}\label{4order}
i \partial_t u + \Delta^2 u - |u|^{2 \sigma}u = 0,
\end{equation}
in much the same way, the inhomogeneous nonlinear Schr\"odinger (INLS) equation
\begin{equation}\label{INLS}
i\partial_tu +\Delta u + |x|^{-b} |u|^{2\sigma} u=0,
\end{equation}
may be seen as an inhomogeneous version of the standard NLS equation. Equation \eqref{4order} was introduced by Karpman \cite{karpman} and Karpman-Shagalov \cite{karpmanshagalov} to take into account the role of small fourth-order dispersion terms in the propagation of intense laser beams in a bulk medium with a Kerr nonlinearity.


Let us start by observing  if $u$ is a solution of \eqref{PVI} so is $u_{\lambda}$ given by
\begin{align}
u_\lambda(x,t)=\lambda^{\frac{4-b}{2\sigma}}u(\lambda x,\lambda^4t), \quad \lambda>0.
\end{align}
 In addition, a straightforward  computation gives
\begin{align}
\|u_\lambda(t)\|_{\dot H^s}=\lambda^{s-s_c}\|u(t)\|_{\dot H^s}.
\end{align}
where $s_c=\frac{N}{2}-\frac{4-b}{2\sigma}$ is the critical Sobolev index. If $s_c = 0$ (or $\sigma = \frac{4-b}{N}$) the IVP \eqref{PVI} is known as  mass-critical or $L^2$-critical; if $s_c=2$ (or $\sigma =\frac{4-b}{N-4}$) it is called energy-critical or $\dot{H}^2$-critical; finally, the problem is known as mass-supercritical and energy-subcritical (also called intercritical) if $0<s_c<2$ (or $\tfrac{4-b}{N}<\sigma<4^*$), where
\begin{equation}
4^*=\begin{cases}\label{def4*}
\tfrac{4-b}{N-4} \quad \textnormal{if}\quad N\geq 5,\\
\infty \quad \textnormal{if}\quad N=1,2,3,4.
\end{cases}
\end{equation}

For solutions in  $H^2$ it is not difficult to see that we have the conservation of mass $M[u]$ and energy $E[u]$ defined by
\begin{equation}\label{mass}
M\left[u(t) \right] = \int |u(t)|^2 dx,
\end{equation}
and
\begin{equation}\label{energy}
E\left[u(t) \right] = \frac{1}{2}\int |\Delta u(t)|^2 dx - \frac{1}{2 \sigma+2} \int |x|^{-b}|u(t)|^{2 \sigma+2} dx.
\end{equation}

Recently, the second and third authors in \cite{GUZPAS} studied the initial value problem \eqref{PVI}. They established local well-posedness in $H^2$ for $N\geq 3$, $0<b<\min\{\frac{N}{2},4\}$ and $\min\left\{\frac{1-b}{N},0\right\}<\sigma<4^*$. Also, they proved global well-posedness in the mass-subcritical  and mass-critical cases in $H^2$, that is, $\min\left\{\frac{1-b}{N},0\right\}<\sigma\leq \frac{4-b}{N}$. In the mass-supercritical and energy-subcritical cases, the authors shown the small data global existence  under the same assumptions on $b$ for dimensions $N\geq 8$ and $N=3,4$, where the local well-posedness results were obtained. The cases $N=\{5,6,7\}$ were also studied however with extra restrictions on the parameters $b$ and $\sigma$. To be more precise we recall the local well-posedness result in $H^2$, which we will use below (see \cite[Theorem 1.2]{GUZPAS}).\\

\noindent {\bf Theorem A.}
{\it Assume $N\geq3$, $0<b<\min\left\{\frac{N}{2},4\right\}$, and $\max\left\{0,\frac{1-b}{N}\right\}<\sigma<4^*$. If $u_0 \in H^2$, then there exist $T=T(\|u_0\|_{H^2},N,\sigma,b)>0$ and a unique solution of \eqref{PVI} satisfying
$$
u \in C\left([-T,T];H^2 \right) \bigcap L^q\left([-T,T];H^{2,r}    \right),
$$
where ($q,r$) is any $B$-admissible pair\footnote{See Section $2$  for the definitions of $B$-admissible and $\dot{H}^{s_c}$-biharmonic admissible pairs.}.}\\

In this paper we are interested in studying local/global well-posedness for \eqref{PVI} in $\dot{H}^{s_c}\cap \dot{H}^2$, with $0\leq s_c<2$. Moreover, we study some dynamical properties of the blow-up solutions to \eqref{PVI} with initial data in $\dot{H}^{s_c}\cap \dot{H}^2$ for the intercritical regime, i.e., $0<s_c<2$.

Our first result concerns the local well-posedness of \eqref{PVI} in $\dot{H}^{s_c}\cap \dot{H}^2$ with $0\leq s_c<2$.  We only consider  $0<s_c<2$ because the case $\dot{H}^{0}\cap \dot{H}^2=H^2$  corresponds to  Theorem A.

\begin{thm}\label{LWP} Let $N\geq5$, $0<b<\min\{\tfrac{N}{2},4\}$ and $\max\left\{\tfrac{4-b}{N},\tfrac{1}{2}\right\}<\sigma<4^*$.  For any $u_0 \in \dot{H}^{s_c}\cap \dot{H}^2 $,  there exist $T=T(\|u_0\|_{\dot{H}^{s_c}\cap \dot{H}^2}, N,\sigma,b) > 0$ and a unique solution u of \eqref{PVI} satisfying
$$
u\in C\left([-T,T];\dot{H}^{s_c}\cap \dot{H}^2 \right)\bigcap L^q\left([-T,T];\dot{H}^{s_c,p}\cap \dot{H}^{2,p} \right)\bigcap L^a\left([-T,T]; L^r \right),$$
for any pairs ($q,p$) B-admissible and $(a,r)$ $\dot H^{s_c}$-biharmonic admissible. 
\end{thm}

As in the $H^2$-theory developed in \cite{GUZPAS}, the proof of Theorem \ref{LWP} relies on the fixed point argument combined with the Strichartz estimates related to the linear problem. In some sense our arguments extend to \eqref{PVI} the strategy presented in \cite{cazenaveweissler}, where the authors studied the NLS equation (see also \cite[Proposition 1.2]{guo2013note}. However, here  the additional restriction  $1<2\sigma$ appears because in our argument we need to estimate $\Delta ( |x|^{-b}|u|^{2\sigma}u)$; moreover we use an auxiliary space $L^a_tL^r_x$, which will be important to define the metric space where we work with. It is worth mentioning that Theorem \ref{LWP} only holds for $N\geq 5$; this condition appears in view of our nonlinear estimates. On the other hand, it holds in the defocusing case, that is, if we replace the sign ``--'' in front of the nonlinearity in \eqref{PVI} by ``+''. We also point out that similar results as in Theorem \ref{LWP} were established for \eqref{4order} and the fractional NLS equation in \cite{dinh2} and \cite{dinh1}, respectively.

In the sequel we will be concerned with global well-posedness results. To do that, we first prove a Gagliardo-Nirenberg type inequality and use it to establish sufficient conditions for global existence.

\begin{thm}\label{GNU}
Let $N\geq 1$, $0<b<\min\{N,4\}$, $0<\sigma<4^*$ and $2\leq p<\frac{(2\sigma+2)N}{N-b}$, then for any  $u\in \dot H^2\cap L^{p}$ we have
\begin{align}\label{GNsc}
\int_{\Real^N} |x|^{-b} |u(x)|^{2\sigma+2}\,dx\leq K_{opt}\|\Delta u\|_{L^2}^{\frac{2\sigma(s_c-s_p)+2(2-s_p)}{2-s_p}}\|u\|_{L^{p}}^{\frac{2\sigma(2-s_c)}{2-s_p}},
\end{align}
 where $s_p=\frac{N}{2}-\frac{N}{p}$,
 \begin{align}
K_{opt}=\left(\frac{\sigma(s_c-s_p)+(2-s_p)}{\sigma(2-s_c)}\right)^{\frac{(4-b)(p-\sigma_c)}{2p(2-s_p)}}\left(\frac{(\sigma+1)(2-s_p)}{\sigma(s_c-s_p)+(2-s_p)}\right)\|V\|_{L^p}^{-\frac{8\sigma-(p-2)(4-b)}{4-2s_p}}\label{J}
\end{align}
and $V$ is a solution to the elliptic equation
\begin{align}\label{elptcpc}
-\Delta^2 V+|x|^{-b}|V|^{2\sigma}V-|V|^{p-2}V=0
\end{align}
with minimal $L^{p}$-norm.
\end{thm}

To prove Theorem \ref{GNU} we use a variational approach. We follow the strategy in \cite{W_Nonl}, where the author established the optimal constant in the standard Gagliardo-Nirenberg inequality. In our case, the main tool used in the proof is the compact embedding  of $ L^p\cap \dot H^2$ into the  weighted Lebesgue space $L^{2\sigma+2}(|x|^{-b}dx)$ (see Section $4$). In the limiting case $b=0$ and $p=2$ the best constant in Theorem \ref{GNU} was already established in \cite{fib2002}.

As an immediate consequence we obtain the following result, which correspond to the cases $p=2$ and $p=\sigma_c$ in Theorem \ref{GNU}.
\begin{coro}\label{GN} Under the same assumptions of Theorem \ref{GNU}, we have
\begin{itemize}
\item [(i)] ${\displaystyle \int_{\Real^N} |x|^{-b} |u(x)|^{2\sigma+2}\,dx}\leq K_{opt}\|\Delta u\|_{L^2}^{\tfrac{N\sigma+b}{2}}\|u\|_{L^{2}}^{\tfrac{4-b-\sigma(N-4)}{2}},
$\;where  $$K_{opt}=\left(\frac{N\sigma +b}{4-b-\sigma(N-4)}\right)^{\frac{-b-N\sigma}{4}}\tfrac{2\sigma+2}{\|V\|_{L^2}^{2\sigma}}$$  
and $V$ is a solution with minimal $L^2$-norm of 
\begin{align}\label{elptcpc1}
-\Delta^2 V+|x|^{-b}|V|^{2\sigma}V-V=0.
\end{align}
\item [(ii)] ${\displaystyle \int_{\Real^N} |x|^{-b} |u(x)|^{2\sigma+2}\,dx}\leq K_{opt}\|\Delta u\|_{L^2}^{2}\|u\|_{L^{\sigma_c}}^{2\sigma}, 
$ where\footnote{Recall that $\sigma_c=\tfrac{2N\sigma}{4-b}$.}\;\; $$K_{opt}=\frac{\sigma+1}{\|V\|^{2\sigma}_{L^{\sigma_c}}}$$
and $V$ is a solution with minimal $L^{\sigma_c}$-norm of 
\begin{align}\label{elptcpc2}
-\Delta^2 V+|x|^{-b}|V|^{2\sigma}V-|V|^{\sigma_c-2}V=0.
\end{align}
\end{itemize}
\end{coro}

Next we state our results concerning  global existence. The first one establishes sufficient conditions for global existence in $H^2$.

\begin{thm}\label{global}
Assume $N\geq 3$, $\frac{4-b}{N}<\sigma< 4^*$ and $0<b<\min\{\tfrac{N}{2},4\}$. Suppose $u_0 \in H^2$  and let $u(t)$ be the corresponding local solution of \eqref{PVI} according to Theorem A. Let $Q$ be a solution of \eqref{elptcpc1} with minimal $L^2$-norm. If
\begin{equation}\label{cond1}
E[u_0]^{s_c} M[u_0]^{2-s_c}<E[Q]^{s_c}M[Q]^{2-s_c}
\end{equation}
and
\begin{equation}\label{cond3}
\|\Delta u_0\|_{L^2}^{s_c}\|u_0\|_{L^2}^{2-s_c}<\|\Delta Q\|_{L^2}^{s_c}\|Q\|_{L^2}^{2-s_c},
\end{equation}
then $u(t)$ is a global solution in $H^2$. In addition,
\begin{equation}\label{cond4}
\|\Delta u(t)\|_{L^2}^{s_c}\|u(t)\|_{L^2}^{2-s_c}<\|\Delta Q\|_{L^2}^{s_c}\|Q\|_{L^2}^{2-s_c}.
\end{equation}
\end{thm}

The strategy to prove Theorem \ref{global} follows the one  introduced by Holmer-Roudenko \cite{HRasharp} in order to study global solutions for the $3D$ cubic NLS equation in the energy space. We also point out that in  \cite{Farah_well}, Farah showed a similar result for the inhomogenoeus NLS equation \eqref{INLS}.

\begin{rem}
Under assumption \eqref{cond1}, if we replace \eqref{cond3} by $\|\Delta u_0\|_{L^2}^{s_c}\|u_0\|_{L^2}^{2-s_c}>\|\Delta Q\|_{L^2}^{s_c}\|Q\|_{L^2}^{2-s_c}$,  we believe that radial solutions of \eqref{PVI} must blow up in finite time. This will be issue for future investigations. For the biharmonic NLS equation \eqref{4order} a result in  this direction was established in \cite{boulen}.
\end{rem}

Our second result gives a sufficient condition in order to deduce that a maximal solution in $\dot H^{s_c}\cap \dot H^2$ is indeed a global one.

\begin{thm}\label{GWPC}
Assume $N\geq 5$, $\max\{\tfrac{4-b}{N},\tfrac{1}{2}\}<\sigma<\tfrac{4-b}{N-4}$ and $0<b<\min\{\tfrac{N}{2},4\}$. Suppose $u_0\in \dot H^{s_c}\cap \dot H^2$, $0< s_c<2$, and let $u(t)$ be the corresponding solution of \eqref{PVI}, according to Theorem \ref{LWP}, with  maximal time of existence, say, $T^*>0$. If 
\begin{equation}\label{assumV}
    \sup_{t\in [0,T^*)}\|u(t)\|_{L^{\sigma_c}}<\|V\|_{L^{\sigma_c}},
\end{equation}
 where $V$ is a solution of \eqref{elptcpc2} with minimal $L^{\sigma_c}$-norm, then $u(t)$ exists globally in  time.
\end{thm}

As we will see below, the proof of Theorem \ref{GWPC} follows as an application of part (ii) in Corollary \ref{GN}.

\begin{rem}
The restriction on the dimension $N\geq 3$ in Theorem \ref{global} and $N\geq5$ in Theorem \ref{GWPC} comes from the local well-posedness theory presented in Theorem A and Theorem \ref{LWP}, respectively. The additional restriction $1<2\sigma$ also comes from Theorem \ref{LWP}. Thus, if one are able to drop these restrictions in the local well-posedness theory then Theorems \ref{global} and \ref{GWPC} also holds without these assumptions.
\end{rem}

Finally, we consider the phenomenon of $L^{\sigma_c}$-norm concentration. This phenomenon, sometimes also called \textit{weak concentration} to differ from the concentration in $L^2$ (see \cite{Sulem}), concerns the concentration in the Lebesgue space $L^{\sigma_c}$ of a solution $u(t) \in \dot H^{s_c}\cap \dot H^2$ that blows up in finite. First of all, in view of Theorem \ref{GWPC}, if we suppose that $u(t)$ blows up in finite time $T^*>0$, then 
\begin{align}
\sup_{t\in [0,T^*)}\|u(t)\|_{L^{\sigma_c}}\geq \|V\|_{L^{\sigma_c}}.
\end{align}
 In addition, either $\|u(t)\|_{\dot H^{s_c}}\to\infty$ or $\|u(t)\|_{\dot H^{2}}\to \infty$, as $t\to T^*$. Here, we will assume that the blowing up solutions satisfy
\begin{align}\label{condbounded}
\sup_{t\in [0,T^*)}\|u(t)\|_{\dot H^{s_c}}<\infty,
\end{align} 
so that we must have
$$
\lim_{t\to T^*}\|u(t)\|_{\dot H^{2}}=\infty.
$$
This kind of blowing up solutions are known in the literature as being of type II.
For the intercritical NLS equation they were studied, for instance, in \cite{MR_Bsc}.
Note that in the case $s_c=0$,  \eqref{condbounded} holds in view of the conservation of the mass. Our concentration result reads as follows.


\begin{thm}\label{concentration} Assume that $N\geq 5$, $\max\{\tfrac{4-b}{N},\tfrac{1}{2}\}<\sigma<\tfrac{4-b}{N-4}$ and $0<b<\min\{0,\tfrac{N}{2}\}$. Let $u_0\in \dot{H}^{s_c}\cap\dot{H}^2 $ be such that the corresponding solution $u$ to \eqref{PVI} blows up in finite time $T^*>0$ satisfying
\eqref{condbounded}.
Let $\lambda (t)>0$ be a function satisfying
\begin{align}\label{assumplam}
\lim_{t\to T^*}\lambda(t)\| u(t)\|_{\dot H^2}^{\frac{1}{2-s_c}}=+  \infty.
\end{align}
 Then
\begin{align}
\liminf_{t\to T^*} \int_{|x|\leq \lambda(t)}|u(x,t)|^{\sigma_c}\,dx\geq \|V\|^{\sigma_c}_{L^{\sigma_c}},
\end{align}
where $V$ is a solution to the elliptic equation \eqref{elptcpc} with minimal $L^{\sigma_c}$-norm.
\end{thm}

The main ideas to prove Theorem \ref{concentration}  comes from \cite{guo2013note}, where the author showed a similar result for the NLS equation. However, in \cite{guo2013note} it was used a profile decomposition theorem to overcome the loss of compactness. In our case, since equation in \eqref{PVI} is not invariant by translations we will take the advantage of the compactness embedding presented in Proposition \ref{WSC}. This strategy has already been applied to the INLS equation in \cite{cardosoetal}.

The rest of the paper is organized as follows. In section \ref{sec2}, we introduce some notations and give a review of the Strichartz estimates. In Section \ref{sec3}, we  establish local well-posedness in $\dot{H}^{s_c}\cap \dot{H}^2 $. In Section \ref{sec4}, we first establish a Gagliardo-Nirenberg type inequality and use it to show the global well-posendess theory in both $H^2$ and  $\dot{H}^{s_c}\cap \dot{H}^2$, $0<s_c<2$ spaces. Section \ref{CNC} is devoted to show Theorem \ref{concentration}. At last, we present an appendix where we set down some remarks concerning global well-posedness and concentration in the critical space $\dot H^{s_c}$.

\section{Notation and Preliminaries}\label{sec2}
We begin introducing the notation used throughout the paper and list some useful inequalities.  We use $c$ to denote various constants that may vary line by line. Let $a$ and $b$ be positive real numbers, the
notation $a \lesssim b$ means that there exists a positive constant $c$ such that $a \leq cb$. Given a real number $r$, we use $r+$ to denote $r+\varepsilon$ for some $\varepsilon>0$ sufficiently small. For a subset $A\subset \mathbb{R}^N$, $A^C=\mathbb{R}^N \backslash A$ denotes the complement of $A$.

For $N\geq 1$ the number $2^*$ is such that
\begin{align}\label{def2*}
2^*=
\begin{cases}
\frac{2N}{N-4},\,\,N\geq 5\\
+\infty,\,\,\,\,N=1,2,3,4.
\end{cases}
\end{align}

The norm in the usual Sobolev spaces $H^{s,p}=H^{s,p}(\mathbb{R}^N)$ is defined by $\|f\|_{H^{s,p}}:=\|J^sf\|_{L^p}$,
where  $J^s$ stands for the Bessel potential of order $-s$, given via Fourier transform by $\widehat{J^s f}=(1+|\xi|^2)^{\frac{s}{2}}\widehat{f}$. If $p=2$ we denote $H^{s,2}$ simply by $H^s$. 

Let 
$$
\dot{\mathcal{S}}(\Real^N):=\{f\in \mathcal{S}(\Real^N); (D^\alpha\widehat{f})(0)=0, \,\mbox{for all}\, \alpha\in\mathbb{N}^N  \},
$$
where $\mathcal{S}(\Real^N)$ denotes the Schwarz space. Endowed with the topology of $\mathcal{S}(\Real^N)$, $\dot{\mathcal{S}}(\Real^N)$ becomes a locally convex space. By $\dot{\mathcal{S}}'(\Real^N)$ we denote the topological dual of $\dot{\mathcal{S}}(\Real^N)$, which can be identified with the factor space $\mathcal{S}'(\Real^N)/\mathcal{P}$ with $\mathcal{P}$ representing the collection of all polynomials of the form $\sum a_\alpha x^\alpha$, $\alpha\in\mathbb{N}^N$ (see \cite{Triebel}, page 237). In particular, elements in $\dot{\mathcal{S}}'(\Real^N)$ may be considered as elements in ${\mathcal{S}}'(\Real^N)$ modulo a polynomial. Given $s\in\Real$, the homogeneous spaces $\dot{H}^s=\dot{H}^s(\Real^N)$ may be defined in the following way
$$
\dot{H}^s(\Real^N):=\{ f\in\dot{\mathcal{S}}'(\Real^N); \|f\|_{\dot{H}^s}:=\|D^sf\|_{L^2}<\infty \},
$$
where $D^s=(-\Delta)^{s/2}$ is the Fourier multiplier with symbol $|\xi|^s$.
Equivalently, $\dot{H}^s$ may also be defined in view of the Littlewood-Paley decomposition. Moreover, $\dot{H}^s$ are reflexive Banach spaces (see \cite[Proposition 1.19]{Wang} for the completeness).  The reflexibility follows taking into account that the map $T:\dot{H}^s(\Real^N)\to L^2(\Real^N)$ defined by $Tf=(-\Delta)^{s/2}f$ is an isometry from $\dot{H}^s(\Real^N)$ onto a closed subspace of $L^2(\Real^N)$. 

For $m\in\mathbb{N}$, $m\geq1$, we also define the space
$$
\dot{W}^m(\mathbb{R}^N):= \{ f\in L^1_{loc}(\Real^N); D^\alpha f\in L^2(\Real^N)\, \mbox{with} \,\alpha\in\mathbb{N}^N, |\alpha|=m \}
$$
and set
$$
\|f\|_{\dot{W}^m}:=\sum_{|\alpha|=m}\|D^\alpha f\|_{L^2}.
$$
Identifying functions that differ by a polynomial the spaces $\dot{H}^m$ and $\dot{W}^m$ are identical (see \cite[Theorem 3.13]{stev}). So, in what follows we will not distinguish these two spaces. In particular the space $C_0^\infty(\Real^N)$ is dense in $\dot{H}^m\cap L^p$, $1\leq p<\infty$ (see \cite[Theorem 7]{Haj}).

By  $L^{p}_b=L^{p}_b(\Real^N)$, $1\leq p<\infty$, we denote the space of all measurable functions $u$ such that
\begin{align}
\|u \|_{L^{p}_b}:=\left(\int |x|^{-b}|u|^{p}\,dx\right)^{\frac{1}{p}}<\infty.
\end{align}
Thus, $L^{p}_b$ is nothing but the weighted Lebesgue space $L^{p}(|x|^{-b}dx)$.

If $X$ is a Banach space and $I\subset \mathbb{R}$ an interval; the mixed norms in the spaces $L^q_{I}X$, $1\leq q\leq \infty$, of a function $f=f(x,t)$ is defined as
$$
\|f\|_{L^q_{I}X}=\left(\int_I\|f(\cdot,t)\|^q_{X}dt\right)^{\frac{1}{q}},
$$
with the standard modification if $q=\infty$. Moreover, if $I=\mathbb{R}$ we shall use the notation $\|f\|_{L_t^qX}$.

\ We now recall some Strichartz type estimates associated to the linear biharmonic Schr\"odinger propagator. We say the pair $(q,r)$ is biharmonic Schr\"odinger admissible ($B$-admissible for short) if it satisfies
\begin{equation*}
\frac{4}{q}=\frac{N}{2}-\frac{N}{r},
\end{equation*}
with
\begin{equation}\label{L2Admissivel}
\begin{cases}
2\leq  r  < \frac{2N}{N-4},\hspace{0.5cm}\textnormal{if}\;\;\;  N\geq 5,\\
2 \leq  r < + \infty,\;  \hspace{0.5cm}\textnormal{if}\;\;\;1\leq N\leq 4.
\end{cases}
\end{equation}
For $s<2$, the pair $(q,r)$ is called $\dot{H}^s$-biharmonic admissible  if 
\begin{equation}\label{CPA1}
\frac{4}{q}=\frac{N}{2}-\frac{N}{r}-s
\end{equation}
with
\begin{equation}\label{HsAdmissivel}
\begin{cases}
\frac{2N}{N-2s} \leq  r  <\frac{2N}{N-4},\;\;\textnormal{if}\;\;  N\geq 5,\\
2 \leq  r < + \infty,\;\; \; \hspace{0.7cm}\textnormal{if}\;\;\;1\leq N\leq 4.
\end{cases}
\end{equation}
Given $s\in \mathbb{R}$ we introduce the Strichartz norm
$$
\|u\|_{B(\dot{H}^{s})}=\sup_{(q,r)\in \mathcal{B}_{s}}\|u\|_{L^q_tL^r_x},
$$
where $\mathcal{B}_s:=\{(q,r);\; (q,r)\; \textnormal{is} \;\dot{H}^s\textnormal{-biharmonic admissible}\}$. For any $(q,r)\in \mathcal{B}_s$, by $(q',r')$ we denote its dual Lebesgue pair, that is, $q'$ and $r'$ are such that $\frac{1}{q}+\frac{1}{q'}=1$ and $\frac{1}{r}+\frac{1}{r'}=1$. This allow us to introduce the dual Strichartz norm
$$
\|u\|_{B'(\dot{H}^{-s})}=\inf_{(q,r)\in \mathcal{B}_{-s}}\|u\|_{L^{q'}_tL^{r'}_x}.
$$
If $s=0$ then $\mathcal{B}_0$ is the set of all $B$-admissible pairs. Thus, $
\|u\|_{B(L^2)}=\sup_{(q,r)\in \mathcal{B}_{0}}\|u\|_{L^q_tL^r_x}$ and $\|u\|_{B'(L^2)}=\inf_{(q,r)\in \mathcal{B}_{0}}\|u\|_{L^{q'}_tL^{r'}_x}$. To indicate the restriction to a time interval $I\subset \mathbb{R}$ we will use  $B(\dot{H}^s;I)$ and $B'(\dot{H}^{-s};I)$. 

We also recall two important inequalities.
\begin{lemma}[\textbf{Sobolev embedding}]\label{SI} Let $s\in (0,+\infty)$ and $1\leq p<+\infty$.
\begin{itemize}
\item [(i)] If $s\in (0,\frac{N}{p})$ then $H^{s,p}(\mathbb{R}^N)$ is continuously embedded in $L^r(\mathbb{R^N})$ where $s=\frac{N}{p}-\frac{N}{r}$. Moreover, 
\begin{equation}\label{SEI} 
\|f\|_{L^r}\lesssim\|D^sf\|_{L^{p}},
\end{equation}
\item [(ii)] If $s=\frac{N}{2}$ then $H^{s}(\mathbb{R}^N)\subset L^r(\mathbb{R^N})$ for all $r\in[2,+\infty)$. Furthermore,
\begin{equation}\label{SEI1} 
\|f\|_{L^r}\lesssim\|f\|_{H^{s}}.
\end{equation}
\end{itemize}
\begin{proof} See \citet[Theorem $6.5.1$]{BERLOF} (see also \citet[Theorem $3.3$]{LiPo15}). 
\end{proof}
\end{lemma}
In particular one has 
\begin{align}\label{SEsc}
\|f\|_{L^p}\lesssim\|f\|_{\dot H^s},\quad\quad\forall f\in \dot H^s(\Real^N),
\end{align}
where $p=\frac{2N}{N-2s}$. Note that $\dot H^{s_c}\hookrightarrow L^{\sigma_c}$ since $\sigma_c=\frac{2N\sigma}{4-b}=\frac{2N}{N-2s_c}.$

\begin{lemma}[\textbf{Hardy-Littlewood inequality}]\label{Hardy}
Let $1 < p \leq q < +\infty$, $N \geq 1$, $0 < s < N$ and $\rho \geq 0$ satistfy the conditions
$$
\rho < \frac{N}{q},\quad s =\frac{N}{p}-\frac{N}{q}+\rho.
$$
Then, for any $u \in H^{s,p}(\mathbb{R}^N)$ we have
$$
\||x|^{-\rho}u\|_{L^q}
\lesssim \|D^su \|_{L^p}.
$$
\end{lemma}
\begin{proof}
See Theorem B* in \cite{stein}.
\end{proof}

\begin{lemma}[\textbf{Fractional Gagliardo-Nirenberg inequality}]\label{GNinequality}
Assume $1<p,p_0,p_1<\infty$,\newline $s,s_1\in\mathbb{R}$, and $\theta\in[0,1]$. Then the fractional Gagliardo-Nirenberg inequality
$$
\| D^s u \|_{L^p} \lesssim \|u\|_{L^{p_0}}^{1-\theta}\|D^{s_1}u\|_{L^{p_1}}^\theta.
$$
holds if and only if
$$
\frac{N}{p}-s=(1-\theta)\frac{N}{p_0}+\theta\left(\frac{N}{p_1}-s_1\right), \quad s\leq \theta s_1.
$$
\end{lemma}
\begin{proof} 
See Corollary 1.3 (page 30) in \cite{Wang}.
\end{proof}
\begin{coro}\label{Cor-importante}
If $\alpha>1$ then we obtain the following estimate
$$
\|\Delta ( |u|^\alpha u)\|_{L^p}\lesssim \|u\|^{\alpha}_{L^{p_1}}\|\Delta u\|_{L^{p_2}},
$$
where $\tfrac{1}{p}=\tfrac{\alpha}{p_1}+\tfrac{1}{p_2}$.
\end{coro}    
\begin{proof}
Observe that, for $\alpha>1$,
\[
|\Delta \left(|u|^\alpha  u\right)|\lesssim ||u|^\alpha\Delta u|+|u|^{\alpha-1}|\nabla u|^2.
\]
The H\"older inequality implies 
$$
\|\Delta \left(|u|^\alpha  u\right)\|_{L^p}\lesssim \|u\|^{\alpha}_{L^{p_1}}\|\Delta u\|_{L^{p_2}} +\|u\|^{\alpha-1}_{L^{p_1}}\|\nabla u\|_{L^{p_3}}^2,
$$
where $\frac{1}{p}=\frac{\alpha}{p_1}+\frac{1}{p_2}=\frac{\alpha-1}{p_1}+\frac{2}{p_3}$. On the other hand, an application of Lemma \ref{GNinequality} (with $s=1$, $s_1=2$ and $\theta=\tfrac{1}{2}$) gives
\[
\|\nabla u\|_{L^{p_3}}
\lesssim\|u\|^{\tfrac{1}{2}}_{L^{p_1}}\|\Delta u\|^{\tfrac{1}{2}}_{L^{p_2}},
\]
where $\frac{1}{p_3}=\frac{1}{2p_1}+\frac{1}{2p_2}$. By noting that  $\frac{1}{p}=\frac{\alpha}{p_1}+\frac{1}{p_2}$ and combining the last two inequalities we obtain the desired.
\end{proof}

\ We now recall the Strichartz estimates related to the linear problem, which are the main tools to show the local and global well-posedness. See for instance Pausader \cite{Pausader2007} (see also \cite{GUZPAS} and the references therein).

\begin{lemma}\label{Lemma-Str}
Let $I\subset\mathbb{R}$ be an interval and $t_0\in I$.
	The following statements hold.
 \begin{itemize}
\item [(i)] \textnormal{(\textbf{Linear estimates})}.
\begin{equation}\label{SE1}
\| e^{it\Delta^2}f \|_{B(L^2;I)} \leq c\|f\|_{L^2},
\end{equation}
\begin{equation}\label{SE2}
\| e^{it\Delta^2}f \|_{B(\dot{H}^s;I)} \leq c \|f\|_{\dot{H}^s}.
\end{equation}
\item[(ii)] \textnormal{(\textbf{Inhomogeneous estimates})}.
\begin{equation}\label{SE3}					 
\left \| \int_{t_0}^t e^{i(t-t')\Delta^2}g(\cdot,t') dt' \right \|_{B(L^2;I) } \leq c\|g\|_{B'(L^2;I)},
\end{equation}
\begin{equation}\label{SE5}
\left \| \int_{t_0}^t e^{i(t-t')\Delta^2}g(\cdot,t') dt' \right \|_{B(\dot{H}^s;I) } \leq c\|g\|_{B'(\dot{H}^{-s};I)}.
\end{equation}
\end{itemize}
\end{lemma} 

We also recall another useful Strichartz estimates for the fourth-order Schr\"odinger equation. 

\begin{prop}\label{estimativanaolinear} Assume $N \geq 3$. Let $I\subset\mathbb{R}$ be an interval and $t_0\in I$. Suppose that  $s \geq 0$ and $u\in C(I,H^{-4})$ is a solution of
$$
u(t)= e^{i(t-t_0)\Delta^2}u(t_0)+i\int_{t_0}^t e^{i(t-t')\Delta^2}F(\cdot,t')dt',
$$
for some function $F\in L^1_{loc}(I,H^{-4})$. Then, for any $B$-admissible pair $(q,r)$, we have
\begin{equation}\label{ESB2}
\left\|D^s u\right\|_{L^{q}_{I}L_x^{r}} \lesssim \left\| D^{s}u(t_0)\right\|_{L^{2}}+\left\|D^{s-1} F\right\|_{L^{2}_{I}L_x^{\frac{2N}{N+2}}}.
\end{equation}
In particular, when $s=2$, 
\begin{equation}\label{EstimativaImportante}
\left\|\Delta u\right\|_{L^{q}_{I}L_x^{r}} \lesssim \left\| \Delta u(t_0)\right\|_{L^{2}}+\left\|\nabla F\right\|_{L^{2}_{I}L_x^{\frac{2N}{N+2}}}.
\end{equation}
\end{prop}
\begin{proof}
    See Proposition 2.3 in \cite{GUZPAS}.
\end{proof}

We end this section with some standard facts.  Recall that
\begin{equation}\label{RIxb}
\||x|^{-b}\|_{L^\gamma(B)}<+\infty\;\;\;\textnormal{if}\;\;\frac{N}{\gamma}-b>0\quad \textnormal{and}\quad \||x|^{-b}\|_{L^\gamma(B^C)}<+\infty\;\;\;\textnormal{if}\;\;\frac{N}{\gamma}-b<0,  
\end{equation}
where here, and throughout the paper, $B$ denotes the unity ball in $\mathbb{R}^N$, that is, $B=\{ x\in \mathbb{R}^N;|x|\leq 1\}$. By setting  $f(z)=|z|^{2\sigma} z$ and $F(x,z)=|x|^{-b}f(z)$, the complex derivative of $f$ is
\begin{equation*}\label{nonli1}
f_z(z)=\sigma|z|^{2\sigma}\;\;\;\;\;\textnormal{and}\;\;\;\; f_{\bar{z}}    (z)=\sigma|z|^{2\sigma-2}z^2. \end{equation*}
Moreover, for any $z,w\in \mathbb{C}$, we have 
\begin{equation}\label{nonlinearity}
|F(x,z)-F(x,w)|\lesssim |x|^{-b}\left( |z|^{2\sigma}+ |w|^{2\sigma} \right)|z-w|.
\end{equation}
\section{Local well-posedness in $\dot H^{s_c}\cap \dot H^2$, $0<s_c<2$}\label{sec3}
In this section we show the local well-posedness in $\dot H^{s_c}\cap \dot H^2$, $0<s_c<2$ (Theorem \ref{LWP}). The proof follows from a contraction mapping argument, which is based on the Strichartz estimates. To do that, we first establish suitable estimates on the nonlinearity $F(x,u)=|x|^{-b}|u|^{2\sigma} u$ in the Strichartz norms. In view of the function $|x|^{-b}$ in the nonlinearity, in
order to obtain the nonlinear estimates, we frequently need to divide them inside and outside the unit ball. 

\begin{lemma}\label{lema1} Let $N\geq 5$ and $0<b<\min\{\frac{N}{2},4\}$. If $\max\{\frac{2-b}{4},\frac{4-b}{N}\}<\sigma<\frac{4-b}{N-4}$ then the following statement holds  
\begin{align}\label{H1Sl2}
\|\nabla F(x,u) \|_{L^{2}_IL_x^{\frac{2N}{N+2}}}\lesssim T^{\theta_1} \|\Delta u\|^{2\sigma+1}_{B(L^2;I)}+T^{\theta_2}\|D^{s_c} u\|^{2\sigma}_{B(L^2;I)}\| \Delta u\|_{B(L^2;I)},
\end{align}
where $I=[0,T]$ and $\theta_1,\theta_2 >0$.
\end{lemma}
\begin{proof} 
We divide the estimate on $B$ and $B^C$ so that
\begin{equation*}
 \|\nabla F(x,u) \|_{L^{2}_IL_x^{\frac{2N}{N+2}}} \leq C_1+C_2, \end{equation*}
with $C_1= \left \|\nabla F(x,u) \right \|_{L^{2}_IL_x^{\frac{2N}{N+2}}(B)}$ and $C_2=\left \|\nabla F(x,u)\right \|_{L^{2}_IL_x^{\frac{2N}{N+2}}(B^C)}$. First we estimate $C_1$. Indeed, the H\"older and Sobolev inequalities imply
\begin{equation}\label{L1C10}
\begin{split}
\hspace{0.5cm} C_{1}& \leq \| |x|^{-b} \|_{L^\gamma(B)} T^{\frac{1}{q_1}} \|u\|^{2\sigma}_{L^q_IL_x^{2\sigma r_1}}   \| \nabla u \|_{L^q_IL_x^{r_2}}  +T^{\frac{1}{q_1}} \|\nabla(|x|^{-b})\|_{L^d(B)}\|u\|^{2\sigma +1}_{L^q_IL^{(2\sigma+1)e}_x} \\
&\lesssim  \| |x|^{-b} \|_{L^\gamma(B)} \| \Delta u\|^{2\sigma +1}_{L^q_IL^r_x} + \||x|^{-b-1}\|_{L^d(B)}\| \Delta u\|^{2\sigma +1}_{L^q_IL^r_x},
\end{split}
\end{equation}
where
\[
\begin{cases}
\frac{N+2}{2N}=\frac{1}{\gamma}+\frac{1}{r_1}+\frac{1}{r_2}=\frac{1}{d}+\frac{1}{e}\\ 
2=\frac{N}{r}-\frac{N}{2\sigma r_1}=\frac{N}{r}-\frac{N}{(2\sigma+1)e}\;,\;\;\;\;r<\frac{N}{2}\\ 
1=\frac{N}{r}-\frac{N}{r_2}\\
\frac{1}{2}=\frac{1}{q_1}+\frac{2\sigma+1}{q},
\end{cases}
\]
which is equivalent to 
\begin{equation}\label{L1C22}
\begin{cases} 
\frac{N}{\gamma}=\frac{N}{d}-1=\frac{N}{2}-\frac{N(2\sigma+1)}{r}+4\sigma +2,\\ 
\frac{1}{q_1}=\frac{1}{2}-\frac{2\sigma+1}{q}.
\end{cases}
\end{equation}
We now need to check that $\frac{1}{q_1}>0$ and $\||x|^{-b}\|_{L^\gamma(B)}$ and $\||x|^{-b-1}\|_{L^d(B)}$ are finite, i.e., $\frac{N}{\gamma}>b$ and $\frac{N}{d}>b+1$, respectively, by \eqref{RIxb}.  To this end, we choose ($q,r$) given by
\begin{equation}\label{BA}
r=\frac{2N(N+4-2b)}{N^2-2bN+16}\;\;\;\textnormal{and}\;\;\;q=\frac{2(N+4-2b)}{N-4}.
\end{equation}
It is easy to see that $(q,r)$ is $B$-admissible and since $b<\frac{N}{2}$ we deduce $r<\frac{N}{2}$. In addition, the hypothesis $\sigma<\frac{4-b}{N-4}$ yields $\frac{N}{\gamma}-b=\frac{N}{d}-b-1>0$ and $
\frac{1}{q_1}=\frac{4-b-\alpha(N-4)}{N+4-2b}>0.
$
Therefore, $C_1\lesssim T^{\theta_1}\|\Delta u\|^{2\sigma+1}_{B(L^2;I)}$ with  $\theta_1=\frac{1}{q_1}$.

We now estimate $C_2$. Note that
$$
|\nabla F(x,u)|\lesssim |x|^{-b}\left( | |u|^{2\sigma} \nabla u|+|u|^{2\sigma} ||x|^{-1} u|\right).
$$
Here we need to divide the proof according to $b < 2$ and $b \geq 2$.

\ {\bf Case $b<2$.} Combining Lemma \ref{Hardy}, H\"older's inequality and Sobolev's embedding one has  
\[
\begin{split}
C_2 &\lesssim  \left \|\||x|^{-b}\|_{L^\gamma(B^C)}  \|u\|^{2\sigma}_{L_x^{2\sigma r_1}}   \| \nabla u \|_{L_x^{\frac{2N}{N-2}}}\right\|_{L^2_I}\\
&\lesssim  T^{\frac{1}{q_1}} \||x|^{-b}\|_{L^\gamma(B^C)} \| D^{s_c} u \|^{2\sigma}_{L^q_IL_x^{r}} \|\Delta u\|_{L^\infty_IL_x^2} ,
\end{split}
\]
where we have used that $\||x|^{-1} u\|_{L_x^{\frac{2N}{N-2}}} \lesssim \|\nabla u\|_{L_x^{\frac{2N}{N-2}}}$ (see\footnote{Here we are using that $N\geq5$.} Lemma \ref{Hardy}) and
\begin{equation}\label{LGR1} 
\frac{N+2}{2}=\frac{N}{\gamma}-\frac{N 2\sigma}{r}+2\sigma s_c-\frac{N-2}{2}\qquad \textnormal{and}\qquad \frac{1}{q_1}=\frac{1}{2}-\frac{2\sigma}{q}.
\end{equation}
Hence, choosing ($q,r)$ $B$-admissible defined by (since $b<2$)
$$
q=\frac{8 \sigma}{2-b}\qquad \textnormal{and}\qquad r=\frac{4\sigma N}{2\sigma N-4+2b},
$$
we\footnote{It is easy to see that $2<r<\frac{N}{s_c}$. Furthermore, in view of $\sigma >\frac{2-b}{4}$ one has $r<\frac{2N}{N-4}$.} have that $\frac{N}{\gamma}-b<0$, that is, $\||x|^{-b}\|_{L^\gamma(B^C)}$ is finite and $\frac{1}{q_1}>0$. Hence, $C_2\lesssim T^{\theta_2}\|D^{s_c} u\|^{2\sigma}_{B(L^2;I)}\|\Delta u\|_{B(L^2;I)}$ with $\theta_2=\frac{1}{q_1}$.

\ {\bf Subcase $b\geq 2$.} Arguing as in the previous case and choosing $(q,r)=(\infty,2)$ we deduce that $\frac{1}{q_1}=\frac{1}{2}>0$ and  $\frac{N}{\gamma}-b=-2<0$, so $\||x|^{-b}\|_{L^\gamma(B^C)}<\infty$. This leads to $C_2\lesssim T^{\theta_2}\|D^{s_c} u\|^{2\sigma}_{B(L^2;I)}\|\Delta u\|_{B(L^2;I)}$ with $\theta_2=\frac{1}{2}$. Thus, the proof of the lemma is completed.
\end{proof}

\begin{lemma}\label{lema2} Let $N\geq 5$ and $0<b<\min\{\frac{N}{2},4\}$. If $\max\{\tfrac{4-b}{N},\frac{1}{2}\}<\sigma<\tfrac{4-b}{N-4}$ then the following inequality holds  
\begin{align}
\left \|D^{s_c} F(x,u)\right \|_{B'(L^2;I)}\lesssim (T^{\theta_1}+T^{\theta_2})\|\Delta u\|_{B(L^2;I)}\|D^{s_c} u\|^{2\sigma}_{B(L^2;I)},
\end{align}
where $I=[0,T]$ and $\theta_1,\theta_2 >0$.
\end{lemma}

\begin{proof}		
Here we use the $B$-admissible pair $(q_\varepsilon,r_\varepsilon)=(\tfrac{4}{2-\varepsilon},\tfrac{2N}{N-4+2\varepsilon})$, where $\varepsilon>0$ is sufficiently small.  From the Sobolev embedding we have
\begin{align}\label{sobolev}
\left\|D^{s_c} F(x,u)\right\|_{L^{r'_\varepsilon}_x}\lesssim \left\|\Delta F(x,u)\right\|_{L^{p^*}_x},
\end{align} 
where $p^*=\frac{2N\sigma}{8\sigma+4-b-2\varepsilon\sigma}$. 
Observe that 
$$
|\Delta F(x,u)|\lesssim |x|^{-b} \left (\Delta(|u|^\alpha |u|)+|x|^{-2}||u|^\alpha u|+|x|^{-b}|x|^{-1}|\nabla (|u|^\alpha u)|\right)
$$
and by Lemma \ref{Hardy} one has  $\left\||x|^{-2}(|u|^{2\sigma} u)\right\|_{L_x^\beta} \lesssim \left\|\Delta (|u|^{2\sigma} u)\right\|_{L_x^\beta}$,\;\; $\left\||x|^{-1}\nabla(|u|^{2\sigma} u)\right\|_{L_x^\beta} \lesssim \left\|\Delta (|u|^{2\sigma} u)\right\|_{L_x^\beta}$ for any  $1<\beta<\frac{N}{2}$. Let $A$ denotes either $B$ or $B^C$. Applying the H\"older inequality, the Sobolev embedding and Corollary \ref{Cor-importante} (since $2\sigma>1$) we get
\begin{equation}\label{a1}
\begin{split}
\left\|\Delta F(x,u)\right\|_{L_I^{q'_\varepsilon}L_x^{p^*}}&\lesssim  \left\|\||x|^{-b}\|_{L_x^\gamma(A)}\|\Delta \left(|u|^{2\sigma}  u\right)\|_{L_x^\beta}\right\|_{L_I^{q'_\varepsilon}}\\
&\lesssim \left\|\||x|^{-b}\|_{L_x^\gamma(A)} \|u\|^{2\sigma}_{L_x^{r_1}} \| \Delta u \|_{L_x^{r}}\right\|_{L_I^{q'_\varepsilon}}\\
&\lesssim T^{\frac{1}{q_1}} \||x|^{-b}\|_{L_x^\gamma(A)} \|D^{s_c}u\|^{2\sigma}_{L^{q}_IL_x^{r}} \| \Delta u \|_{L^{q}_IL_x^{r}},
\end{split}
\end{equation}
where (using the value of $p^*$ defined above)
\begin{equation}\label{a2}
\frac{N}{\gamma}=\frac{8\sigma +4-b-2\varepsilon \sigma}{2\sigma}-\frac{N}{\beta},\quad \frac{1}{\beta}=\frac{2\sigma}{r_1}+\frac{1}{r}, \;\quad\; s_c=\frac{N}{r}-\frac{N}{r_1}\quad\textnormal{and}\;\quad\; \frac{1}{q'_\varepsilon}=\frac{1}{q_1}+\frac{2\sigma+1}{q}.
\end{equation}
This is equivalent to 
\begin{align}\label{sistema} 
\left\{\begin{array}{rcl}\vspace{0.1cm}
\frac{N}{\gamma}-b&=&\frac{ 4-b+2\sigma^2N-2\varepsilon\sigma}{2\sigma}-\frac{N(2\sigma+1)}{r},\\ \vspace{0.1cm}
\frac{1}{q_1}&=&\frac{1}{q'_\varepsilon}-\frac{2\sigma+1}{q}.
\end{array}\right.
\end{align}
We need to find $(q,r)$ $B$-admissible such that  $\left\||x|^{-b}\right\|_{L^\gamma(A)}$ is finite, $r<\frac{N}{s_c}$ and $\frac{1}{q_1}>0$. Indeed, if $A=B$ we choose $(q,r)$ $B$-admissible defined by 
\begin{align}\label{parBadmissible}
r=\frac{2\sigma N(2\sigma+1)}{2\sigma^2N+4-b-4\varepsilon \sigma}\,\,\,\,\mbox{ and }\,\,\,\,q=\frac{8\sigma(2\sigma+1)}{\sigma N-4+b+4\varepsilon \sigma},
\end{align}
so\footnote{It is easy to check that $\frac{4}{q}=\frac{N}{2}-\frac{N}{r}$ and if $\frac{4-b}{N}<\sigma <\frac{4-b}{N-4}$, then $2<r<\frac{2N}{N-4}$.} by using \eqref{sistema} we obtain $\frac{N}{\gamma}-b=\varepsilon>0$, i.e., $|x|^{-b}\in L^\gamma (B)$. Moreover, \begin{align}
\theta_1=\frac{1}{q_1}=\frac{1}{q'_\varepsilon}-\frac{2\sigma+1}{q}=\frac{4-b-\sigma(N-4)-2\varepsilon \sigma}{8\sigma},
\end{align}
which is positive since $\sigma<\tfrac{4-b}{N-4}$ and $\varepsilon>0$ is small. 

On the other hand, if $A=B^C$ we choose $(q,r)$ defined by
$$
r=\frac{2\sigma N(2\sigma+1)}{2\sigma^2N+4-b}\,\,\,\,\mbox{ and }\,\,\,\,q=\frac{8\sigma(2\sigma+1)}{\sigma N-4+b},
$$
which gives, from \eqref{sistema}, $\tfrac{N}{\gamma}-b=-\varepsilon<0$ and implies that $\||x|^{-b}\|_{L^\gamma(B^C)}$ is finite and 
$$
\theta_2=\frac{1}{q_1}=\tfrac{4-b-\sigma(N-4-2\varepsilon)}{8\sigma}>0.
$$
To complete the proof we need to verify that $r<\frac{N}{s_c}$ and $\beta <\frac{N}{2}$. Indeed, note that $r<\frac{N}{s_c}$ is equivalent to $8-2b- 4\varepsilon\sigma-\sigma(N-8+2b)>0$. Clearly this is true if $N-8+2b\leq 0$, otherwise if $N-8+2b>0$ then the facts that $\sigma<\tfrac{4-b}{N-4}$ and $b<\tfrac{N}{2}$ imply the desired. In addition, by using the first relation in \eqref{a2} and the value of $\gamma$, we obtain \begin{equation*}
\beta\;=\;
\begin{cases}
\frac{2N\sigma}{8\sigma +4-b-2\sigma b},\;\;\; \qquad \textnormal{if}\;\;A=B^C\\
\frac{2N\sigma}{8\sigma +4-b-2\sigma b-4\varepsilon\sigma},  \quad \textnormal{if}\;\;A=B.
\end{cases}    
\end{equation*}
Observe now that $\beta<\frac{N}{2}$ is equivalent to $4-b+2\sigma(2-b)-4\varepsilon\sigma>0$. For one hand, if $b\geq 2$ then this inequality clearly holds. On the other hand, if $b<2$ then the facts that $b<\frac{N}{2}$ and $\sigma<\frac{4-b}{N-4}$ yield the desired.

The proof of the lemma then follows from \eqref{sobolev} and \eqref{a1}.
\end{proof}

\begin{lemma}\label{lemmacontraction} Let $N\geq 5$ and $0<b<4$.  If $\frac{4-b}{N}<\sigma<\frac{4-b}{N-4}$ then the following statement holds 
\begin{align}
\left \|\chi_{A} |x|^{-b}|u|^{2\sigma} v\right \|_{B'(\dot H^{-s_c};I)} \lesssim (T^{\theta_1}+T^{\theta_2})\|\Delta u\|^{\theta}_{L^{\infty}_IL^2_x}\|u\|^{2\sigma-\theta}_{B(\dot H^{s_c};I)}\|v\|_{B(\dot H^{s_c};I)},
\end{align}
where $A$ is either $B$ or $B^C$, $\theta_1,\theta_2>0$ and $\theta\in (0,2\sigma) $ is small enough.
\end{lemma}
\begin{proof}
Let $(\tilde a,r)$ be any $\dot H^{-s_c}$-biharmonic admissible pair. The Sobolev embedding  and H\"older's inequality yield
\begin{align}\label{E}	
\left\||x|^{-b}|u|^{2\sigma}v \right\|_{L^{\tilde a'}_IL_x^{r'}(A)} &\leq \left\| \left\| |x|^{-b} \right\|_{L^\gamma(A)}  \|u\|^{\theta}_{L_x^{\theta r_1}} \|u\|^{2\sigma-\theta}_{L_{x}^{r}}  \| v \|_{L_x^{r}}\right\|_{L^{\tilde a'}_I} \nonumber  \\
&\lesssim  \left\||x|^{-b}\right\|_{L^{\gamma}(A)}\left\| \| \Delta u \|^{\theta}_{L_x^{2}}\|u\|^{2\sigma-\theta}_{L^r_x}\| v \|_{L_x^{r}}  \right\|_{L^{\tilde a'}_I}\nonumber \\
&\lesssim  T^{\frac{1}{q_1}}\left\||x|^{-b}\right\|_{L^{\gamma}(A)}\| \Delta u \|^{\theta}_{L_t^{\infty}L_x^{2}}\|u\|^{2\sigma-\theta}_{L^a_tL^r_x}\| v \|_{L_t^{a}L_x^{r}},
\end{align}
where
\begin{equation}\label{E1}
\left\{\begin{array}{rcl}\vspace{0.1cm}
\frac{1}{r'}&=&\frac{1}{\gamma}+\frac{1}{r_1}+\frac{2\sigma-\theta}{r}+\frac{1}{r},\\ \vspace{0.1cm}
2&=&\frac{N}{2}-\frac{N}{\theta r_1}, \\ \vspace{0.1cm}
\frac{1}{\tilde a'}&=&\frac{1}{q_1}+\frac{2\sigma-\theta}{a}+\frac{1}{a}.
\end{array}\right.
\end{equation}
We need to check that  $\tfrac{1}{q_1}>0$ and $|x|^{-b}\in L^\gamma (A)$, i.e., $\tfrac{N}{\gamma}>b$ if $A=B$ and $\tfrac{N}{\gamma}<b$ if $A=B^C$. In fact, by assuming that $(a,r)$ is $\dot H^{s_c}$-admissible, the conditions \eqref{E1} are equivalent to
\begin{equation}
\left\{\begin{array}{rcl}\vspace{0.1cm}
\frac{N}{\gamma}-b&=&N-b-\frac{N\theta}{2}+2\theta-\frac{N(2\sigma +2-\theta)}{r},\\ \vspace{0.1cm}
\frac{2}{q_1}&=&2-s_c-\frac{2(2\sigma+2-\theta)}{a}.
\end{array}\right.
\end{equation} 
If $A=B$, we set
\begin{align}
a=\frac{2(2\sigma+2-\theta)}{2-s_c-\varepsilon}\,\,\,\,\mbox{ and }\,\,\,\,r=\frac{2N(2\sigma+2-\theta)}{(N-2s_c)(2\sigma+2-\theta)-4(2-s_c-\varepsilon)}
\end{align}
for $0<\varepsilon<\frac{\theta(2-s_c)}{2}$ small enough\footnote{The value of $\widetilde{a}$ is given by $\widetilde{a}=\tfrac{2(2\sigma -\theta+2)}{s_c(2\sigma-\theta)+2+s_c-\varepsilon}$. Moreover, since $s_c<2$ and $a>\frac{4}{2-s_c}$ we have $2<r<\frac{2N}{N-4}$.}. 
Hence,
\begin{align}
\frac{2}{q_1}&=2-s_c-\frac{2(2\sigma+2-\theta)}{a}=\varepsilon>0\\
\frac{N}{\gamma}-b&=N-b-\frac{N\theta}{2}+2\theta-\frac{N(2\sigma+2-\theta)}{r}=\theta(2-s_c)-2\varepsilon>0.	\end{align}

On the other hand, if $A=B^C$, we set
\begin{align}
a=\infty\,\,\,\,\mbox{ and }r=\frac{2N}{N-2s_c}=\frac{2\sigma N}{4-b},
\end{align}  
which gives\footnote{Here $\widetilde{a}=\tfrac{2}{s_c}$. Observe that since $0<s_c<2$ we obtain $2<r<\frac{2N}{N-4}$.},
\begin{align}
\frac{2}{q_1}&=2-s_c>0\quad \textnormal{and}\quad \frac{N}{\gamma}-b=-(2-\theta)(2-s_c)<0.
\end{align}
This completes the proof of the lemma.
\end{proof}

We now show our local well-posedness result.

\begin{proof}[{Proof of Theorem \ref{LWP}}]	
For any ($q,p$) $B$-admissible and $(a,r)$ $\dot H^{s_c}$- biharmonic admissible, we define 
$$
X= C\left([-T,T];\dot{H}^{s_c}\cap \dot{H}^2 \right)\bigcap L^q\left([-T,T];\dot{H}^{s_c,p}\cap \dot{H}^{2,p}\right)\bigcap L^a\left([-T,T]; L^r \right),$$ and 
\begin{equation*}\label{NHs} 
\|u\|_T=\|\Delta u\|_{B\left(L^2;[-T,T]\right)}+\|D^{s_c} u\|_{S\left(L^2;[-T,T]\right)}+\|u\|_{B\left(\dot H^{s_c};[-T,T]\right)}.
\end{equation*}
We shall prove that for some $T>0$ the operator $G=G_{u_0}$ defined by
\begin{equation}\label{OPERATOR} 
G(u)(t)=e^{it\Delta^2}u_0+i \int_0^t e^{i(t-t')\Delta^2}|x|^{-b}|u(t')|^{2\sigma} u(t')dt'
\end{equation}
is a contraction on the complete metric space 
\begin{equation*}
S(a,T)=\{u \in X : \|u\|_T\leq a \}
\end{equation*}
with the metric 
$$
d_T(u,v)=\|u-v\|_{B\left(\dot H^{s_c};[-T,T]\right)}.
$$

Indeed,  from   Lemma \ref{Lemma-Str} and \eqref{EstimativaImportante} we deduce
\begin{align}\label{NSD} 
\|\Delta G(u)\|_{B\left(L^2;[-T,T]\right)}&\leq c\|\Delta u_0\|_{L^2}+\left\|\nabla F(x,u) \right\|_{L^2_{[-T,T]}L_x
^{\frac{2N}{N+2}}}
\end{align}
\begin{align}
\|D^{s_c} G(u)\|_{B\left(L^2;[-T,T]\right)}&\leq c \|D^{s_c} u_0\|_{L^2}+ c\|D^{s_c} F(x,u)\|_{B'\left(L^2;[-T,T]\right)}
\end{align}
and
\begin{align}
\|G(u)\|_{B\left(\dot H^{s_c};[-T,T]\right)}\leq c\|u_0\|_{\dot H^{s_c}}+c\|\chi_{B}F(x,u)\|_{B'\left(\dot H^{-s_c};[-T,T]\right)}+c\|\chi_{B^c}F(x,u)\|_{B'\left(\dot H^{-s_c};[-T,T]\right)}
\end{align}
where $F(x,u)=|x|^{-b}|u|^{2\sigma} u$. Without loss of generality we consider only the case $t > 0$. It follows from  Lemma \ref{lema1}, Lemma \ref{lema2} and Lemma \ref{lemmacontraction} that
\begin{align}\label{F} 
\|\nabla F(x,u) \|_{L^{2}_IL_x^{\frac{2N}{N+2}}}\lesssim T^{\theta_1} \|\Delta u\|^{2\sigma+1}_{B(L^2;I)}+T^{\theta_2}\|D^{s_c} u\|^{2\sigma}_{B(L^2;I)}\| \Delta u\|_{B(L^2;I)}
\end{align}
\begin{align}
\|D^{s_c} F\|_{B'(L^2;I)}&\leq c(T^{\theta_1}+T^{\theta_2}) \| \Delta u \|_{B(L^2;I)}\|D^{s_c}u\|^{2\sigma}_{B(L^2;I)}
\end{align}
and
\begin{align}
\|\chi_{B} F\|_{B'(\dot H^{-s_c};I)}+\|\chi_{B^c}F\|_{B'(\dot H^{-s_c};I)}&\leq c(T^{\theta_1}+T^{\theta_2}) \| \Delta u \|^{\theta}_{L^{\infty}_IL^2_x}\|u\|^{2\sigma+1-\theta}_{B(\dot H^{s_c};I)},
\end{align}
where $I=[0,T]$ and $\theta_1,\theta_2>0$. Hence, if $u \in S(a,T)$ then
\begin{align}\label{esa}
\|G(u)\|_T 
&\leq c \|u_0\|_{\dot{H}^{s_c}\cap \dot{H}^2 }+c (T^{\theta_1}+T^{\theta_2}) a^{2\sigma+1}.
\end{align}
Choosing $a=2c\|u_0\|_{\dot{H}^{s_c}\cap \dot{H}^2 }$ and $T>0$ such that 
\begin{equation}\label{CTHs} 
c a^{2\sigma}(T^{\theta_1}+T^{\theta_2}) < \frac{1}{4},
\end{equation}
we get $G(u)\in S(a,T)$, which implies that $G$ is well defined on $S(a,T)$. To show that $G$ is a contraction we use \eqref{nonlinearity} and Lemma \ref{lemmacontraction} to deduce
\begin{align*}
d_T(G(u),G(v))&\leq c\|\chi_B\left( F(x,u)-F(x,v)\right)\|_{B'\left(\dot H^{-s_c};I\right)}+c\|\chi_{B^c} \left(F(x,u)-F(x,v)\right)\|_{B'\left(\dot H^{-s_c};I\right)}\\ 
&\leq c (T^{\theta_1}+T^{\theta_2})\left(\|\Delta u\|^{\theta}_{L^{\infty}_tL^2_x}\|u\|_{S(\dot H^{s_c};I)}^{2\sigma-\theta}+\|\Delta v\|^{\theta}_{L^{\infty}_tL^2_x}\|v\|_{B(\dot H^{s_c};I)}^{2\sigma-\theta}\right)\|u-v\|_{S(\dot H^{s_c};I)}\\
&\leq c (T^{\theta_1}+T^{\theta_2})\left(\|u\|^{2\sigma}_T+\|v\|^{2\sigma}_T\right)d_T(u,v),
\end{align*}
and thus if $u,v\in S(a,T)$, then 
$$
d_T(G(u),G(v))\leq c (T^{\theta_1}+T^{\theta_2})a^{2\sigma} d_T(u,v).
$$
Therefore, by using \eqref{CTHs} we have that $G$ is a contraction on $S(a,T)$ and by the fixed point theorem one has a unique fixed point $u \in S(a,T)$ of $G$. 
The proof is thus completed.
\end{proof}

\section{Gagliardo-Nirenberg inequality and Global solutions}\label{sec4}

This section is devoted to prove Theorems \ref{GNU} and \ref{global}. We start by studying the relation between the best optimal constant in \eqref{GNsc} and the solutions of equation \eqref{elptcpc}.

\subsection{The ground states}

First of all, we recall that inequality
\begin{align}\label{GNsc1}
\int_{\Real^N} |x|^{-b} |u(x)|^{2\sigma+2}\,dx\leq C\|\Delta u\|_{L^2}^{\frac{2\sigma(s_c-s_p)+2(2-s_p)}{2-s_p}}\|u\|_{L^{p}}^{\frac{2\sigma(2-s_c)}{2-s_p}},
\end{align}
holds\footnote{Recalling that $s_p=\frac{N}{2}-\frac{N}{p}$.} for some constant $C>0$, provided  $0<b<\min\{N,4\}$, $0<\sigma<4^*$ and $1\leq p\leq\frac{(2\sigma+2)N}{N-b}$ (see \cite[page 1516]{Lin}. Our intention is then to prove that for $p\geq2$ the optimal constant one can place in \eqref{GNsc1} is exactly the one given in \eqref{J}. To do so, we follow the strategy as in \cite{W_Nonl}, where the author proved a similar result for the standard Gagliardo-Nirenberg inequality.

For  $f\in \dot H^{2}\cap L^{p}$, we define the ``Weinstein functional'' as
	\begin{align}
J(f)=\frac{\|\Delta f\|_{L^2}^{\frac{2\sigma(s_c-s_p)+2(2-s_p)}{2-s_p}}\|f\|_{L^{p}}^{\frac{2\sigma(2-s_c)}{2-s_p}}}{\|f\|_{L_b^{2\sigma+2}}^{2\sigma+2}}
\end{align}
and set 
\begin{equation}\label{Jdef}
    J=\displaystyle\inf_{f\in \dot H^{2}\cap L^{p}, f\neq 0} J(f).
\end{equation}
 From \eqref{GNsc1} it is clear that $J(f)$ is well defined for any $f\in \dot H^{2}\cap L^{p}$, $J>0$ and the optimal constant in \eqref{GNsc1} is $1/J$. Our first task now is to show that the infimum of $J$ is indeed attained. As we will see below, this infimum is directly connected  to the solutions of the equation
\begin{align}\label{groundstate}
-\Delta^2 \phi +|x|^{-b}|\phi|^{2\sigma}\phi=|\phi|^{p-2}\phi.
\end{align}

Here, by a solution of \eqref{groundstate} we mean a critical point of the functional
\begin{equation}
    I(f)=\frac{1}{2}\int |\Delta f|^2\,dx-\frac{1}{2\sigma+2}\int|x|^{-b}|f|^{2\sigma+2}dx+\frac{1}{p}\int |f|^p\,dx.
\end{equation}

In the following lemma we obtain  Pohozaev-type identities which are satisfied by any solution of \eqref{groundstate}.
\begin{lemma}[Pohozaev-type identities]\label{pohozaev}
	Let $\phi\in \dot H^2\cap L^{p}$ be a solution to \eqref{groundstate}. Then the following identities hold
		\begin{align}\label{l6}
	\int |x|^{-b}|\phi|^{2\sigma+2}\,dx=\int |\Delta \phi|^2\,dx+\int |\phi|^{p}\,dx
	\end{align}
	and
		\begin{align}\label{l5}
	\frac{N-b}{2\sigma+2}\int |x|^{-b}|\phi|^{2\sigma+2}\,dx=\left(\frac{N}{2}-2\right)\int |\Delta \phi|^2\,dx+\frac{N}{p}\int |\phi|^{p}\,dx.
	\end{align}
		In particular,
	\begin{align}\label{h1sc}
	\int |\Delta \phi|^2\,dx=\frac{N(2\sigma+2)-(N-b)p}{2\sigma p(2-s_c)}\int |\phi|^{p}\,dx
	\end{align}
	and
	\begin{align}\label{epsc}
	\int |x|^{-b}|\phi|^{2\sigma+2}\,dx=\left(\frac{N(2\sigma+2)-(N-b)p}{2\sigma p(2-s_c)}+1\right)\int |\phi|^{p}\,dx.
	\end{align}	
\end{lemma}
\begin{proof}
First we note that \eqref{h1sc} and \eqref{epsc} follow easily from \eqref{l6} and \eqref{l5}. So, it suffices to establish \eqref{l6} and \eqref{l5}. Below we proceed formally, but by using a standard  argument we may see that we have the regularity and decay to justify all calculations.

To obtain \eqref{l6} we just multiply \eqref{groundstate} by $\bar\phi$, integrate on $\mathbb{R}^N$ and use integration by parts. On the other hand, by multiplying \eqref{groundstate} by $x\cdot \nabla \bar \phi$ and taking the real part, we obtain
	\begin{align}\label{l1}
	-\mbox{Re}\,\int \Delta^2 \phi\, x\cdot \nabla \bar \phi\,dx+\mbox{Re}\,\int |x|^{-b}|\phi|^{2\sigma}\phi\, x\cdot \nabla \bar \phi\,dx=\mbox{Re}\,\int |\phi|^{p-2}\phi\, x\nabla \bar u\,dx.
	\end{align}
	Integrating by parts, 
	\begin{align}
	\int \Delta^2 \phi\,x\cdot \nabla \bar{\phi}\,dx=4\int|\Delta\phi|^2\,dx-\int 
	\nabla\phi\cdot x\Delta^2\bar\phi\,dx-N\int |\Delta \phi|^2\,dx,
	\end{align}
	and thus
	\begin{equation}\label{l2}
	    \mbox{Re}\,\int \Delta^2 \phi\,x\cdot \nabla \bar \phi\,dx=\left(2-\frac{N}{2}\right)\int |\Delta \phi|^2\,dx.
	\end{equation}

Integration by parts also yields
	\begin{align}
	\int |x|^{-b}|\phi|^{2\sigma}\phi\,x\cdot \nabla \phi\,dx&=\sum_{j=1}^N\int |x|^{-b}|\phi|^{2\sigma}\phi x_j\partial_j\bar \phi\,dx=-\sum_{j=1}^{N}\int\partial_j(|x|^{-b}|\phi|^{2\sigma}\phi x_j)\bar \phi\,dx\\
	&=b\int |x|^{-b}|\phi|^{2\sigma+2}\,dx - 2\sigma\mbox{Re}\,\int |x|^{-b}|\phi|^{2\sigma}\bar \phi\, x\cdot\nabla \phi\,dx\\
	&\quad-N\int |x|^{-b}|\phi|^{2\sigma+2}\,dx-\int |x|^{-b}|\phi|^{2\sigma}\bar \phi\,x\cdot \nabla \phi\,dx,
	\end{align}
	which implies
	\begin{align}\label{l3}
	\mbox{Re} \int |x|^{-b}|\phi|^{2\sigma}\phi\,x\cdot \nabla \bar \phi\,dx=-\frac{N-b}{2\sigma+2}\int |x|^{-b}|\phi|^{2\sigma+2}\,dx.
	\end{align}
	In a similar fashion, we obtain
	\begin{align}\label{l4}
	\mbox{Re}\,\int |\phi|^{p-2}\phi\,x\cdot \nabla \bar \phi\,dx=-\frac{N}{p}\int |\phi|^{p}\,dx.
	\end{align}
	Using \eqref{l1}, \eqref{l2}, \eqref{l3}, and \eqref{l4} we get \eqref{l5}.
\end{proof}

Under the assumptions of Theorem \ref{GNU}, inequality \eqref{GNsc1} shows that the embedding $\dot{H}^2\cap L^p\hookrightarrow L^{2\sigma+2}_b$ is continuous. Below we show that such a embedding is in fact compact.

\begin{prop}\label{WSC}
Let $N\geq 1$, $0<b<\min\{N,4\}$, $0<\sigma<4^*$ and $1< p<\frac{(2\sigma+2)N}{N-b}$. Then
 the embedding  
\begin{align}
\dot H^2\cap L^{p} \hookrightarrow L^{2\sigma+2}_b 
\end{align}
is compact.
\end{prop}
\begin{proof}
First we recall from Lemma \ref{GNinequality} that
\begin{align}
\dot H^2\cap L^q\hookrightarrow L^r\label{sobolev1sigmac},
\end{align}
where $1<q<r<2^*$ ($2^*$ is given in \eqref{def2*}).

Now, let $\{u_n\}_{n=1}^{+\infty}$ be a bounded sequence in $\dot H^2\cap L^{p} $. We want to show that, up to a subsequence, $\{u_n\}_{n=1}^{+\infty}$ converges in $L^{2\sigma+2}_b$. Define $r=\frac{(2\sigma+2)N}{N-b+\eta}$, with $0<\eta\ll1$ to be chosen later\footnote{Note that for $\eta$ small enough and $0<\sigma<4^*$ we have $p< r<2^*$.}. Thus, in view of \eqref{sobolev1sigmac}, $\{u_n\}_{n=1}^{+\infty}$ is also bounded in  $L^r$. Since $L^r$ is a reflexive Banach space, there exists $u\in L^{r}$ such that, up to a subsequence, $u_n\rightharpoonup u$ in $L^{r} $ as $n\to \infty$. Defining $w_n=u_n-u$, we will show that 
\begin{align}
\int|x|^{-b}|w_n|^{2\sigma+2}\,dx\to 0,
\end{align}
as $n\to \infty.$ The main idea is to split the integral on the ball $B(0,R)$ and on $\Real^N\backslash B(0,R)$, for appropriate $R>0$.
	
Let us start by recalling that for any $R>0$ and $\alpha>N$, we have 
\begin{align}\label{intR}
\int_{\Real ^N\backslash B(0;R)}|x|^{-\alpha}\,dx\lesssim \frac{1}{R^{\alpha-N}}.
\end{align} 
By  choosing $\eta$ such that $\eta<b$ and $N-b+\eta>0$ we may define $\gamma_1=\frac{N}{N-b+\eta}$ and $\gamma_1'=\frac{N}{b-\eta}$.	Thus, by H\"older's inequality, we infer
\begin{align}\label{foradabola}
\int_{\Real^{N}\backslash B\left(0;R\right)}|x|^{-b}\left|w_n\right|^{2\sigma+2}\,dx\leq \left(\int_{\Real^N\backslash B(0,R)}|x|^{-b\gamma_1'}dx\right)^{\frac{1}{\gamma_1'}}\left(\int_{\Real^N\backslash B(0,R)}|w_n|^{(2\sigma+2)\gamma_1}\,dx\right)^{\frac{1}{\gamma_1}}.
\end{align}
 Since $b\gamma_1'>N$ and $(2\sigma+2)\gamma_1=r$, given $\varepsilon>0$ from \eqref{intR} and \eqref{foradabola} we may choose $R>0$ large enough such that
		\begin{align}\label{inteps}
		\int_{\Real^N\backslash B(0,R)}|x|^{-b}|w_n|^{2\sigma+2}\,dx<\varepsilon.
		\end{align}
		Now, we estimate the integral over the ball $B(0,R)$. Initially, let us show that $\{u_n\}_{n\in \mathbb{N}}$ is uniformly bounded in $ H^{2}(B(0,R))$. Indeed, from the standard Gagliardo-Nirenberg inequality it suffices to show that $\{u_n\}_{n\in \mathbb{N}}$ as well as the derivatives of order two are bounded in $L^2(B(0,R))$. Note if $1<p<2$ then \eqref{sobolev1sigmac} implies
		\begin{equation}\label{eq1}
		    \|u_n\|_{L^2(B(0,R))}\leq  \|u_n\|_{L^2}\lesssim \| u_n\|_{\dot H^2\cap L^p}.
		\end{equation}
		Also, if $2\leq p<\frac{(2\sigma+2)N}{N-b}$ then the embedding $L^p(B(0,R))\hookrightarrow L^2(B(0,R))$ gives
		\begin{equation}\label{eq2}
	 \|u_n\|_{L^2(B(0,R))}\lesssim  \|u_n\|_{L^p(B(0,R))}\lesssim \| u_n\|_{L^{p }}.
		\end{equation}
		Moreover,
		\begin{equation}\label{eq3}
		\sum_{|\alpha|=2}\|D^\alpha u_n\|_{L^2(B(0,R))}\leq 	\sum_{|\alpha|=2}\|D^\alpha u_n\|_{L^2(\Real^N)}\lesssim \|\Delta u_n\|_{L^2(\Real^N)}\lesssim \|u_n\|_{\dot{H}^2}.
		\end{equation} 
	From, \eqref{eq1}, \eqref{eq2}, and \eqref{eq3} we conclude that $\{u_n\}_{n\in \mathbb{N}}$ is  bounded in $ H^{2}(B(0,R))$.

	Hence, from the Rellich-Kondrashov theorem (see, for instance, \cite{liebloss})
		and the fact that $w_n\rightharpoonup 0$ in $L^{r}(\mathbb{R}^N)$, we deduce that, 	up to a subsequence,
		\begin{align}\label{convstrong}
		w_n\to 0\mbox{ in } L^q\left(B(0,R)\right)\quad \textnormal{for all} \quad q\in (1,2^*).
		\end{align}
	Next, for $\eta<N-b$ the numbers $\gamma_2=\frac{N}{N-b-\eta}$ and $\gamma_2'=\frac{N}{b+\eta}$ are such $\frac{1}{\gamma_2}+\frac{1}{\gamma_2'}=1$. Hence, by H\"older's inequality we have
		\begin{align}
		\int_{B(0,R)}|x|^{-b}\left|w_n\right|^{2\sigma+2}\,dx&\leq \left(\int_{B(0,R)}|x|^{-b\gamma_2'}\,dx\right)^{\frac{1}{\gamma_2'}}\left(\int_{B(0,R)}\left|w_n\right|^{(2\sigma+2)\gamma_2}\,dx\right)^{\frac{1}{\gamma_2}}.
		\end{align}
	It is easy to see that $b\gamma_2'<N$ and (for $\eta$ small enough) $1<(2\sigma+2)\gamma_2<2^*$.	So, in view of \eqref{convstrong}, given $\varepsilon>0$, we obtain  for  $n$ large enough,
		\begin{align}\label{intps1}
		\int_{B(0,R)}|x|^{-b}\left|w_n\right|^{2\sigma+2}\,dx<{\varepsilon}.
		\end{align}
	A combination of \eqref{inteps} and \eqref{intps1} completes the proof of Proposition \ref{WSC}.
\end{proof}

With the compactness embedding in Proposition  \ref{WSC} we are able to prove Theorem \ref{GNU}.

\begin{proof}[Proof of Theorem \ref{GNU}]
 Let $\{f_n\}_{n\in \mathbb{N}}$ be a minimizing sequence for \eqref{Jdef}, that is, a sequence of nontrivial functions in $\dot{H}^2\cap L^p$ satisfying $J(f_n)\to J$ as $n\to\infty$. By setting 
 \begin{align}
\mu_n=\frac{\|f_n\|_{L^p}^{\frac{N-4}{2(2-s_p)}}}{\|f_n\|^{\frac{N}{p(2-s_p)}}_{\dot H^2}}\,\,\,\,\,\mbox{ and }\,\,\,\,\theta_n=\left(\frac{\|f_n\|_{L^{p}}}{\| f_n\|_{\dot H^2}}\right)^{\frac{1}{2-s_p}}
\end{align}
and defining $g_n(x)=\mu_nf_n(\theta_nx)$ it is not difficult to check that
\begin{equation}\label{f1}
    \|g_n\|_{L^p}=1, \quad \|g_n\|_{\dot H^2}=1, \quad \mbox{and}\quad J(g_n)=J(f_n)
\end{equation}
and, consequently, $\{g_n\}_{n\in \mathbb{N}}$ is also a minimizing sequence for \eqref{Jdef} and bounded in $\dot H^{2}\cap L^{p}$.

Hence, there exists $g^{*}\in \dot H^{2}\cap L^{p}$ such that, up to a subsequence, $g_n\to g^{*}$ weakly in $\dot H^{2}\cap L^{p}$ and strongly in $L^{2\sigma+2}_b$ (see Proposition \ref{WSC}). From \eqref{f1} we deduce
	\begin{align}
	\|g^*\|_{L^{p}}\leq 
	1\,\,\,\,\,\,\mbox{ and }\,\,\,\,\,\,\|\Delta g^*\|_{L^2}= \|g^*\|_{\dot H^2}\leq 1,
	\end{align}
	which implies
	\begin{align}\label{f2}
	J\leq J(g^*)\leq \frac{1}{\left\|g^*\right\|_{L_b^{2\sigma+2}}^{2\sigma+2}}=\lim_{n\to\infty} \frac{1}{\left\|g_n\right\|_{L_b^{2\sigma+2}}^{2\sigma+2}}=J.
	\end{align}
	Consequently, the inequalities in \eqref{f2} must be equalities and
	\begin{align}\label{f3}
	J=J(g^*)=\frac{1}{\left\|g^*\right\|_{L_b^{2\sigma+2}}^{2\sigma+2}} \,\,\,\,\,\,\,\,\,\,\textnormal{ and }\,\,\,\,\,\,\,\,\,\,\|g^*\|_{L^{p}}=\|\Delta g^*\|_{L^2}=1.
	\end{align}
This shows, in particular, that $g^*\neq 0$ and $g^*$ is a minimizer for the functional $J$.

Next let us prove that, up to a scaling, $g^*$ is a solution of \eqref{elptcpc}. Indeed, since $g^*$ is a minimizer of $J$, we have
\begin{align}
\frac{d}{d\varepsilon} J(g^*+\varepsilon g)|_{\varepsilon=0}=0,\,\,\,\,\mbox{ for all }g \in C^{\infty}_0(\Real^N),
\end{align}
which, using \eqref{f3}, is equivalent to
\begin{align}
-\frac{\sigma(s_c-s_p)+(2-s_p)}{2-s_p}\Delta^2  g^*+(\sigma+1)J|x|^{-b}|g^*|^{2\sigma}g^*=\frac{\sigma(2-s_c)}{2-s_p}|g^*|^{p-2}g^*.
\end{align}
Now, we let $V$ be defined by $g^*(x)=\alpha V(\beta x)$ with
\begin{align}\label{alfa}
\alpha=\left[\frac{AJ}{B^{\frac{4-b}{4}}}\right]^{\frac{4}{(p-2)(4-b)-8\sigma}},\,\,\,\,\,\,\,\,\,\,\,\, \,\,\,\,\,\,\,\,\,\,\,\,\beta=B^{\frac{1}{4}}\left[\frac{AJ}{B^{\frac{4-b}{4}}}\right]^{\frac{p-2}{(p-2)(4-b)-8\sigma}}
\end{align}
and
\begin{align}
A=\frac{(\sigma+1)(2-s_p)}{\sigma(s_c-s_p)+(2-s_p)},\,\,\,\,\,\,\,\,\,\,\,\,\,\,\,\,\,\,\,\,\,\,\,\,B=\frac{\sigma(2-s_c)}{\sigma(s_c-s_p)+(2-s_p)}.
\end{align}
Hence $V$ is a solution of \eqref{elptcpc} and 
\begin{align}\label{Vc}
\|V\|_{L^{p}}^{p}=\frac{\beta^N}{\alpha^{p}}\|g^*\|_{L^{p}}^{p}=\frac{\beta^N}{\alpha^{p}}=B^{\frac{N}{4}}\left(\frac{AJ}{B^{\frac{4-b}{4}}}\right)^{\frac{N(p-2)-4p}{(p-2)(4-b)-8\sigma}}.
\end{align}
Note that this implies
\begin{align}
K_{opt}&=\frac{1}{J}=AB^{-\frac{(4-b)(p-\sigma_c)}{2p(2-s_p)}}\|V\|_{L^p}^{-\frac{8\sigma-(p-2)(4-b)}{4-2s_p}}\\&=\left(\frac{\sigma(s_c-s_p)+(2-s_p)}{\sigma(2-s_c)}\right)^{\frac{(4-b)(p-\sigma_c)}{2p(2-s_p)}}\left(\frac{(\sigma+1)(2-s_p)}{\sigma(s_c-s_p)+(2-s_p)}\right)\|V\|_{L^p}^{-\frac{8\sigma-(p-2)(4-b)}{4-2s_p}},
\end{align}
which completes the proof of the theorem.
\end{proof}

As a consequence of the Theorem \ref{GNU}, we are able to establish the global well-posedness result in Theorem \ref{global}.
\begin{proof}[Proof of Theorem \ref{global}]
As usual, the idea is to get an a priori bound of the local solution in $H^2$. Since, by \eqref{mass} the $L^2$ norm is already conserved, we only need to bound $\|\Delta u(t)\|_{L^2}$. To do so, we will use the conservation of the energy in \eqref{energy} and Corollary \ref{GN} (part (i)). Indeed, first note that
	\begin{align}\label{estimateg1}
	2E(u_0)=2E(u(t))=&\|\Delta u(t)\|_{L^2}^2-\frac{1}{\sigma+1}\left\|u(t)\right\|_{L_b^{2\sigma+2}}^{2\sigma+2}\\ \geq& \|\Delta u(t)\|_{L^2}^2-\frac{K_{opt}}{\sigma+1}\|\Delta u(t)\|_{L^2}^{\sigma s_c+2}\|u(t)\|_{L^{2}}^{\sigma(2-s_c)}.
	\end{align}
By setting $X(t)=\|\nabla u(t)\|_{L^2}^2$ and $B=\frac{K_{opt}}{\sigma+1}M[u_0]^{\frac{\sigma (2-s_c)}{2}}$, the last inequality reads as
	\begin{align}
	    X(t)-BX(t)^{\frac{\sigma s_c+2}{2}}\leq 2E[u_0],\,\,\,\,\,\,\,\,\forall t\in [-T,T],
	\end{align}
	where $T>0$ is the existence time provided by the local theory in Theorem A.
	
	Next, for $x\geq 0$ we define the function $f(x)=x-Bx^{\frac{\sigma s_c+2}{2}}$ . Since $\sigma>\frac{4-b}{N}$ it is easy to see that $f$ has a maximum point at
	\begin{align}
	    x_0=\left(\frac{2}{B(\sigma s_c+2)}\right)^{\frac{2}{\sigma s_c}}
	\end{align}
	with maximum value
	\begin{align}
	   f(x_0)=\frac{\sigma s_c}{\sigma s_c+2}\left(\frac{2}{B(\sigma s_c+2)}\right)^{\frac{2}{\sigma s_c}}.
	\end{align}
	By requiring that 
	\begin{equation}\label{imp}
	   2E[u_0]< f(x_0) \quad \mbox{and}\quad X(0)<x_0,
	\end{equation}
	the continuity of $X(t)$ yields that $X(t)<x_0$, for any $t\in[-T,T]$.
	
	Applying Lemma \ref{pohozaev} we obtain
	\begin{align}
	E[Q]=&\frac{1}{2}\|\Delta Q\|_{L^2}^{2}-\frac{1}{2\sigma+2}\|Q\|_{L_b^{2\sigma+2}}^{2\sigma+2}
	=\left(\frac{N\sigma+b}{4\sigma (2-s_c)}-\frac{1}{\sigma (2-s_c)}\right)\|Q\|^{2}_{L^2}
	=\frac{s_c}{2(2-s_c)}\|Q\|_{L^2}^2.
	\end{align}
	In addition, since
	\begin{align}\label{estimateg2}
	K_{opt}=\left(\frac{\sigma(2-s_c)}{\sigma s_c+2}\right)^{\frac{\sigma s_c}{2}}\left(\frac{2\sigma+2}{\sigma s_c+2}\right)\|Q\|_{L^2}^{-2\sigma},
	\end{align}
	a straightforward calculation gives that  \eqref{cond1} and \eqref{cond3} are equivalent to those relations in \eqref{imp}.
	Furthermore,  \eqref{cond4} is equivalent to $X(t)<x_0$. The proof of the theorem is thus completed.
\end{proof}

\begin{rem}
An argument similar to that in the proof of Theorem \ref{global} also holds in the case $\sigma=\frac{4-b}{N}$. Indeed, since $s_c=0$, as in \eqref{estimateg1} we obtain
\begin{equation}
    2E[u_0]\geq \|\Delta u(t)\|_{L^2}^2\left(1-\frac{K_{opt}}{\sigma+1}\|u_0\|_{L^{2}}^{2\sigma}\right).
\end{equation}
Moreover, because $K_{opt}=(\sigma+1)\|Q\|_{L^2}^{-2\sigma}$ (see \eqref{estimateg2}) it follows that
\begin{equation}\label{estimateg3}
    2E[u_0]\geq \|\Delta u(t)\|_{L^2}^2\left(1- \frac{\|u_0\|_{L^{2}}^{2\sigma}}{\|Q\|_{L^{2}}^{2\sigma}}\right).
\end{equation}
If we assume 
\begin{equation}\label{estimateg4}
    \|u_0\|_{L^2}<\|Q\|_{L^2},
\end{equation}
from \eqref{estimateg3} we get that $\|\Delta u(t)\|_{L^2}$ is uniformly bounded. Thus, under assumption \eqref{estimateg4} we obtain the global well-posedness in $H^2$ in the case $\sigma=\frac{4-b}{N}$. This result complements the one in  \cite[Proposition 1.5]{GUZPAS}, where the authors proved the global well-posedness assuming that $\|u_0\|_{L^2}$ is sufficiently small. In the limiting case $b=0$ this result was also established in \cite[Theorem 4.2]{fib2002}.

\end{rem}

\begin{proof}[Proof of Theorem \ref{GWPC}] Since we are assuming \eqref{condbounded} it is sufficient to show that $\|\Delta u(t)\|_{L^2}$ remains bounded. But, using the conservation of the energy and Corollary \ref{GN} (part (ii)) we deduce
\[
\begin{split}
 2E[u_0]&\geq \|\Delta u(t)\|_{L^2}^2\left(1-\frac{\|u(t)\|_{L^{\sigma_c}}^{2\sigma}}{\|V\|_{L^{\sigma_c}}^{2\sigma}}\right)\\
 &\geq \|\Delta u(t)\|_{L^2}^2\left(1-\frac{\sup_{t\in[0,T^*)}\|u(t)\|_{L^{\sigma_c}}^{2\sigma}}{\|V\|_{L^{\sigma_c}}^{2\sigma}}\right).
\end{split}
\]
Thus, under assumption \eqref{assumV} we obtain the desired and the proof of the theorem is completed. 
\end{proof}

\begin{rem}
Theorem \ref{GWPC} is also true in the case $\frac{1}{2}<\sigma=\frac{4-b}{N}$, in which case we have $\sigma_c=2$ and \eqref{assumV} reduces to \eqref{estimateg4}.
\end{rem}


\section{Concentration in the critical Lebesgue norm}\label{CNC}

In this section, we prove our main result about $L^{\sigma_c}$-norm concentration in the intercritical regime for finite time blow-up solution. 
\begin{proof}[Proof of Theorem \ref{concentration}]
Clearly it suffices to show that 
\begin{align}\label{sufcond}
\liminf_{n\to\infty}\int_{|x|\leq \lambda(t_n)}|u(x,t_n)|^{\sigma_c}\,dx\geq \|V\|^{\sigma_c}_{L^{\sigma_c}},
\end{align}
for any sequence $\{t_n\}_{n\in \mathbb{N}}$ satisfying $t_n\uparrow T^*$, as $n\to\infty$.

To prove this, let $\{t_n\}_{n\in \mathbb{N}}$ be an arbitrary  sequence such that $t_n\uparrow T^*$ and define
\begin{align}
\rho_{n}=\left(\frac{1}{\| u(t_n)\|_{\dot H^2}}\right)^{\frac{1}{2-s_c}}\,\,\,\,\,\,\,\,\mbox{ and }\,\,\,\,\,\,\,v_n(x)=\rho_n^{\frac{4-b}{2\sigma}}u(\rho_n x,t_n).
\end{align}
By using \eqref{condbounded} we may take a positive constant $C$ such that
$$
\|v_n\|_{\dot H^{s_c}}=\|u(t_n)\|_{\dot H^{s_c}}\leq C.
$$
Thus, in view of the embedding   $\dot H^{s_c}\subset L^{\sigma_c}$ we deduce that $\{v_n\}_{n\in \mathbb{N}}$ is bounded in $L^{\sigma_c}$.
Also, a straightforward calculation gives
$$
\|v_n\|_{\dot H^2}^2=\rho_n^{\frac{2(4-b)}{2\sigma}+4-N}\| u(t_n)\|_{\dot H^2}^{2}=\rho_n^{2(2-s_c)}\| u(t_n)\|_{\dot H^2}^{2}=1.
$$
So $\{v_n\}_{n\in \mathbb{N}}$ is a bounded sequence in $L^{\sigma_c}\cap \dot H^{2}$. Consequently, there exists $v^*\in L^{\sigma_c}\cap \dot H^{2}$ such that, up to a subsequence, $v_n\rightharpoonup v^*$ in $ L^{\sigma_c}\cap \dot H^{2}$, as $n\to\infty$. This implies
\begin{align}\label{liminffrac}
\|\Delta v^*\|_{L^2}\leq \liminf_{n\to \infty}\|\Delta v_n\|_{L^2}\,\,\,\,\mbox{ and }\,\,\,\,\|v^*\|_{L^{\sigma_c}}\leq \liminf_{n\to \infty}\|v_n\|_{L^{\sigma_c }}.
\end{align}
Furthermore, by Proposition \ref{WSC} (passing to a subsequence if necessary)
\begin{align}\label{convgweak}
\lim_{n\to\infty}\left\|v_n\right\|_{L_b^{2\sigma+2}}=\left\|v^* \right\|_{L_b^{2\sigma+2}}.
\end{align}

On the other hand, note that
$$
E(v_n)=\frac{1}{2}\|\Delta v_n\|_{L^2}^{2}-\frac{1}{2\sigma+2}\int |x|^{-b}|v_n|^{2\sigma+2}\,dx=\rho_n^{2(2-s_c)}E(u_0)
$$
and since $\rho_n\to 0$ as $n\to \infty$, we have  
\begin{equation}\label{En}
    \displaystyle\lim_{n\to \infty}E(v_n)=0.
\end{equation}
Combining Corollary \ref{GNU} together with \eqref{liminffrac}, \eqref{convgweak} and \eqref{En} one has 
\begin{align}
0=\liminf_{n\to\infty}E(v_n)\geq \frac{1}{2}\|\Delta v^*\|_{L^2}^{2}\left(1-\frac{\|v^*\|_{L^{\sigma_c}}^{2\sigma}}{\|V\|_{L^{\sigma_c}}}\right),
\end{align}
which implies $\|v^*\|_{L^{\sigma_c}}\geq \|V\|^{2\sigma}_{L^{\sigma_c}}$.

Next, the weak convergence $v_n\rightharpoonup v^*$ in $L^{\sigma_c}$, gives that for any $R>0$,
\begin{align}
\liminf_{n\to\infty}\int_{|y|\leq\rho_nR}|u(y,t_n)|^{\sigma_c}\,dy=&\liminf_{n\to\infty}\int_{|x|\leq R}\rho_n^{\frac{\sigma_c(4-b)}{2\sigma}}|u(\rho_n x,t_n)|^{\sigma_c}\,dx\\=&\liminf_{n\to\infty}\int_{|x|\leq R}|v_n(x)|^{\sigma_c}\,dx\geq \int_{|x|\leq R}|v^*|^{\sigma_c}\,dx,
\end{align}
Since assumption \eqref{assumplam} implies $\lambda(t_n)/\rho_n\to \infty$, as $n\to\infty$, we obtain
\begin{align}\label{liminlam}
\liminf_{n\to\infty}\int_{|x|\leq \lambda(t_n)}|u(x,t_n)|^{\sigma_c}\,dx\geq \int_{|x|\leq R}|v^*|^{\sigma_c}\,dx,
\end{align}
for any $R>0$. Finally, from \eqref{liminlam} and the fact that $\|v^*\|_{L^{\sigma_c}}\geq \|V\|^{2\sigma}_{L^{\sigma_c}}$ we obtain \eqref{sufcond},
which completes the proof.
\end{proof}

\section{Appendix}\label{ACR}
In this appendix we give another concentration result for the IBNLS equation.
 First of all, we note that Proposition \ref{WSC} also allows us to obtain a Gagliardo-Niremberg type inequality for functions in $\dot H^{s_c}\cap\dot H^2$. More precisely, we have the following.
\begin{thm}\label{GNUsc}
	Let $N\geq 1$, $0<b<4$, $\frac{4-b}{N}<\sigma< 4^*$ and       $s_c=\frac{N}{2}-\frac{4-b}{2\sigma}$, then for any $f\in \dot H^{s_c}\cap\dot H^2$, we have
	\begin{align}\label{GNisc}
	\int_{\Real^N} |x|^{-b} |f(x)|^{2\sigma+2}\,dx\leq \frac{\sigma+1}{\|W\|_{\dot H^{s_c}}^{2\sigma}}\|\Delta f\|_{L^2}^2\|f\|_{\dot H^{s_c}}^{2\sigma},
	\end{align}
	where $W$ is a solution of the elliptic equation, 
	\begin{align}\label{elptcsc1}
	-\Delta^2 W+|x|^{-b}|W|^{2\sigma}W-(-\Delta)^{s_c}W=0
	\end{align}
with minimal $\dot H^{s_c}$-norm.
\end{thm}
\begin{proof}(Sketch)
	Since the proof is similar to that of Theorem \ref{GNU} we will give only the main steps. Define the functional 
	\begin{align}
	J(f)=\frac{\|\Delta f\|_{L^2}^{2}\|f\|_{\dot H^{s_c}}^{2\sigma}}{\left\|f\right\|_{L_b^{2\sigma+2}}^{2\sigma+2}}
	\end{align}
	and set
	\begin{equation}\label{J11}
	    	J=\displaystyle\inf_{f\in \dot H^{s_c}\cap\dot H^2, f\neq 0}J(f).
	\end{equation}

	  Let $\{f_n\}_{n\in \mathbb{N}}$ be a minimizing sequence for \eqref{J11} and consider the sequence $w_n(x)=\mu_nf_n(\theta_nx)$ with
	\begin{align}
	\mu_n=\frac{\|f_n\|_{\dot H^{s_c}}^{\frac{N-4}{2(2-s_c)}}}{\|\Delta f_n\|_{L^2}^{\frac{4-b}{2\sigma(2-s_c)}}}\,\,\,\,\mbox{ and }\,\,\,\,\theta_n=\left(\frac{\|f_n\|_{\dot H^{s_c}}}{\|\Delta f_n\|_{L^{2}}}\right)^{\frac{1}{2-s_c}}.
	\end{align}
It is easy to check that $\{w_n\}_{n\in \mathbb{N}}$ is also a minimizing sequence satisfying $\|w_n\|_{\dot H^{s_c}}=\|\nabla w_n\|_{L^2}=1$. An application of Proposition \ref{WSC} gives a function $w^*\in \dot H^{s_c}\cap \dot H^2$ such that $w_n\rightharpoonup w^*$ in $\dot H^{s_c}\cap \dot H^2$, as $n\to \infty$, and $J(w^*)=J.$ In addition, $w^*$ must be a solution of the equation
$$
-\Delta^2 w^*-(\sigma+1)J|x|^{-b}|w^*|^{2\sigma}w^*+\sigma(-\Delta)^{s_c}w^*=0.
$$
The scaling  $w^*(x)=\alpha W(\beta x)$ with
$$
\alpha=\left( \frac{\sigma^{\frac{4-b}{2(2-s_c)}}}{(\sigma+1)J} \right)^{\frac{1}{2\sigma}} \quad \mbox{and} \quad \beta=\sigma^{\frac{1}{2(2-s_c)}}
$$
then gives that $W$ is a solution of \eqref{elptcsc1} with minimal $\dot H^{s_c}$-norm and
	$$\|W\|_{\dot H^{s_c}}=\frac{\beta^{\frac{4-b}{2\sigma}}}{\alpha}\|w^*\|_{\dot H^{s_c}}=[(\sigma+1)J]^{\frac{1}{2\sigma}}.$$
	Consequently, for any $f\in \dot H^{s_c}\cap \dot H^2$
	$$\int|x|^{-b}|f(x)|^{2\sigma+2}\leq \frac{1}{J}\|\Delta f\|_{L^2}^2\|u\|_{\dot H^{s_c}}^{2\sigma}=\frac{\sigma+1}{\|W\|_{\dot H^{s_c}}}\|\Delta f\|_{L^2}^2\|f\|_{\dot H^{s_c}}^{2\sigma},$$
	which completes the proof of the theorem.
\end{proof}

With this Gagliardo-Nirenberg type inequality in hand we are able to prove the following.
\begin{thm}\label{global2}
	Let $N\geq5$,  $0<b<\min\{\frac{N}{2},4\}$ and $\max\{\tfrac{4-b}{N},\tfrac{1}{2}\}<\sigma<\frac{4-b}{N-4}$. For $u_0\in \dot H^{s_c}\cap \dot H^2$, let $u(t)$ be the corresponding solution to \eqref{PVI} given by Theorem \ref{LWP} and $T^*>0$ the maximal time of existence. Suppose that $\sup_{t\in [0,T^*)}\|u(t)\|_{\dot H^{s_c}}<\|W\|_{\dot H^{s_c}}$, where $W$ is a solution of the elliptic equation \eqref{elptcsc1} with minimal $\dot H^{s_c}$-norm. Then $u(t)$ exists globally in the time.
\end{thm}

\begin{thm}\label{concentrationsc} Let $N\geq5$, $0<b<\min\{\frac{N}{2},4\}$ and $\max\{\tfrac{4-b}{N},\tfrac{1}{2}\}<\sigma<\frac{4-b}{N-4}$. For $u_0\in \dot{H}^{s_c}\cap\dot{H}^2$, let $u(t)$ be the corresponding solution to \eqref{PVI} given by Theorem \ref{LWP} and assume that it blows up in finite time $T^*>0$ satisfying \eqref{condbounded}.
If $\lambda (t)>0$ is a function satisfying
\begin{align}
\lambda(t)\| u(t)\|_{\dot H^2}^{\frac{1}{2-s_c}}\to  \infty, \,\,\textit{ as }\, t\to T^*,
\end{align}
then
\begin{align}
\liminf_{t\to T^*} \int_{|x|\leq \lambda(t)}|D^{s_c}u(x,t)|^{2}\,dx\geq \|W\|^{2}_{\dot H^{s_c}},
\end{align}
	where $W$ is a minimal $\dot H^{s_c}$-norm solution to the elliptic equation \eqref{elptcsc1}.
\end{thm}

The proofs of Theorems \ref{global2} and \ref{concentrationsc} are similar to those of Theorems \ref{GWPC} and  \ref{concentration}. So we omit the details.

\vspace{0.5cm}
\noindent 
\textbf{Acknowledgments.} A.P. is partially supported by CNPq/Brazil grant 303762/2019-5 and FAPESP/Brazil grant 2019/02512-5.

\newcommand{\Addresses}{{
	\bigskip
	\footnotesize
		
	MYKAEL A. CARDOSO, 
	\textsc{Department of Mathematics, UFPI, Brazil}\par\nopagebreak
	\textit{E-mail address:} \texttt{mykael@ufpi.edu.br}
		
	\medskip
		
	 CARLOS M. GUZM\'AN, \textsc{Department of Mathematics, UFF, Brazil;}\par\nopagebreak
	 \textit{E-mail address:} \texttt{carlos.guz.j@gmail.com}
	
	\medskip
	
	ADEMIR PASTOR, \textsc{Department of Mathematics, IMECC-UNICAMP, Brazil}\par\nopagebreak
		\textit{E-mail address:} \texttt{apastor@ime.unicamp.br}

}}
\setlength{\parskip}{0pt}
\Addresses

\end{document}